\documentclass[preprint,12pt]{elsarticle}

\usepackage{amscd}
\usepackage{amsmath}
\usepackage{amsthm}
\usepackage{amsfonts}
\usepackage{graphicx}
\usepackage{amssymb}
\usepackage{color}
\usepackage{amsbsy}
\usepackage{graphicx}
\usepackage{amssymb,color,amsbsy}

\usepackage{float}
\usepackage{comment}

\usepackage[ruled]{algorithm2e}

\usepackage[matrix,arrow]{xy}
\usepackage{mathabx}   

\DeclareMathAlphabet{\mathpzc}{OT1}{pzc}{m}{it}

\usepackage{pgfpages}
\usepackage{tikz}
\usetikzlibrary{shapes.geometric}
\usetikzlibrary{decorations.pathmorphing}
\usetikzlibrary{arrows}
\usetikzlibrary{backgrounds}
\usetikzlibrary{positioning}
\usetikzlibrary{fit}
\usepackage{caption}
\usepackage{amsthm}

\usepackage{amsmath}

\usepackage[pagebackref=true, bookmarksopen=true,colorlinks=true, linkcolor=red,citecolor=blue]{hyperref}

\usepackage[capitalise]{cleveref}  

\global\long\def\R{\mathbb{R}}%

\global\long\def\ii{\cap}%

\global\long\def\u{\cup}%

\global\long\def\I{\bigcap}%

\global\long\def\U{\bigcup}%

\global\long\def\s{\subset}%

\global\long\def\p{\prime}%

\global\long\def\P{\prime}%

\global\long\def\sm{\sum}%

\global\long\def\ñ{\sim}%

\usepackage{tikz-cd}

\usepackage{tkz-graph}

\usepackage{multicol}

\usepackage{subcaption}

\newtheorem{theorem}{Theorem}[section]

\newtheorem{conjecture}[theorem]{Conjecture}
\newtheorem{corollary}[theorem]{Corollary}

\newtheorem{lemma}[theorem]{Lemma}

\newtheorem{problem}[theorem]{Problem}

\newcommand{\Sachs}[1]{\operatorname{Sachs}(#1)}

\usepackage{comment}
\begin{document}

\begin{abstract}
	
	A graph is almost bipartite if it contains exactly one odd cycle,
	and it is König–Egerváry if $\alpha(G)+\mu(G)$ equals the order of $G$.
	We introduce the class of Bipartite–Almost Bipartite graphs (BAB-graphs),
	defined through a controlled union of a bipartite graph and
	several almost bipartite non–König–Egerváry graphs. This family unifies
	and generalizes the previously studied classes of almost bipartite
	non–König–Egerváry and $R$-disjoint graphs. While an almost
	bipartite non–König–Egerváry graph contains a single odd cycle,
	an $R$-disjoint graph has exactly $k$ pairwise disjoint odd cycles.
	A BAB-graph may contain many odd cycles that are not necessarily disjoint. 
	
	We describe the structure of BAB-graphs by means of the Gallai–Edmonds
	decomposition of the graph and obtain explicit expressions for
	$\text{nucleus}(G)$, $\text{diadem}(G)$, and $\text{ker}(G)$, which
	allow us to extend several known results for the previous classes.
	Moreover, we show that the determinant of the adjacency matrix of
	a BAB-graph can be factorized in terms of the determinants of
	the adjacency matrices of its component graphs. As a consequence, we confirm the conjecture stating the validity of this factorization for $R$-disjoint graphs. Finally,
	we derive combinatorial consequences of these results and establish
	new bounds for $\left|\text{corona}(G)\right|+\left|\text{ker}(G)\right|$.
	
\end{abstract}

\begin{frontmatter} 
	\title{On Bipartite–Almost Bipartite Graphs and the Determinantal Factorization}

	\author[pan,daj]{Kevin Pereyra} 	\ead{kdpereyra@unsl.edu.ar}


	\address[pan]{Universidad Nacional de San Luis, Argentina.} 	\address[daj]{IMASL-CONICET, Argentina.} 
	\date{Received: date / Accepted: date} 
	
\end{frontmatter} %

\section{Introduction}

Let $\alpha(G)$ denote the cardinality of a maximum independent set,
and let $\mu(G)$ be the size of a maximum matching in $G=(V,E)$.
It is known that $\alpha(G)+\mu(G)$ equals the order of $G$,
in which case $G$ is a König--Egerváry graph 
\cite{deming1979independence,gavril1977testing,stersoul1979characterization}.
Various properties of König--Egerváry graphs were presented in 
\cite{bourjolly2009node,jarden2017two,levit2006alpha,levit2012critical}.
It is known that every bipartite graph is a König--Egerváry graph 
\cite{egervary1931combinatorial}. 

Let $\Omega^{*}(G)=\left\{ S:S\textnormal{ is an independent set of }G\right\}$,
$\Omega(G)=\{S:S$ is a maximum independent set of $G\}$,
$\textnormal{core}(G)=\I\left\{ S:S\in\Omega(G)\right\}$ 
\cite{levit2003alpha+}, and 
$\textnormal{corona}(G)=\U\left\{ S:S\in\Omega(G)\right\}$ 
\cite{boros2002number}. 
The number $d_{G}(X)=\left|X\right|-\left|N(X)\right|$ is the 
difference of the set $X\s V(G)$, and 
$d(G)=\max\{d_{G}(X):X\s V(G)\}$ is called the \emph{critical difference} of $G$.
A set $U\s V(G)$ is \emph{critical} if $d_{G}(U)=d(G)$ 
\cite{zhang1990finding}. 
The number $d_{I}(G)=\max\left\{ d_{G}(X):X\in\Omega^{*}(G)\right\}$ 
is called the \emph{critical independence difference} of $G$. 
If a set $X\s\Omega^{*}(G)$ satisfies $d_{G}(X)=d_{I}(G)$, then it is called 
a \emph{critical independent set} \cite{zhang1990finding}. 
Clearly, $d(G)\ge d_{I}(G)$ holds for every graph. 
It is known that $d(G)=d_{I}(G)$ for all graphs 
\cite{zhang1990finding}. 
We define $\textnormal{ker}(G)=\I\{ S:S$ is a critical independent set of $G\}$ 
\cite{levit2012vertices,lorentzen1966notes,schrijver2003combinatorial}.
Let $\textnormal{nucleus}(G)=\I\{S:S$ is a maximum critical independent 
set of $G\}$ \cite{jarden2019monotonic}, and 
$\textnormal{diadem}(G)=\U\{S:S$ is a critical independent set of $G\}$ 
\cite{short2015some}.

A graph $G$ is almost bipartite if it has exactly one odd cycle. 
It is known that $\textnormal{ker}(G)\s\textnormal{core}(G)$ holds for every graph 
\cite{levit2012vertices}. 
Equality holds for bipartite graphs \cite{levit2013critical}, 
for unicyclic non–König–Egerváry graphs \cite{levit2011core}, 
for almost bipartite non–König–Egerváry graphs 
\cite{levit2025almost,levit2024almost}, and 
for $R$-disjoint graphs, which generalize the class of almost bipartite 
non–König–Egerváry graphs \cite{kevin2025RDG}.

In \cite{edmonds1965paths}, Edmonds introduced the following concepts 
relative to a matching $M$ of a graph $G$ and its subgraphs. 
An $M$-blossom of $G$ is an odd cycle of length $2k+1$ with $k$ edges in $M$. 
The vertex not saturated by $M$ in the cycle is called the \emph{base} of the blossom. 
An $M$-stem is an $M$-alternating path of even length (possibly zero) 
connecting the base of the blossom with a vertex not saturated by $M$. 
The base is the only common vertex between the blossom and the stem. 
An $M$-flower is a blossom joined with a stem. 
The vertex not saturated by $M$ in the stem is called the \emph{root} of the flower.

In \cite{stersoul1979characterization}, Sterboul introduced the concept 
of a \emph{posy} for the first time. 
An $M$-posy consists of two vertex-disjoint blossoms joined by an $mm$-$M$-path. 
The endpoints of the path are exactly the bases of the two blossoms. 
There are no internal vertices of the path inside the blossoms.

\begin{theorem}
	\label{sterboul} 
	For a graph $G$, the following properties are equivalent: 
	\begin{itemize}
		\item $G$ is a non–König--Egerváry graph,
		\item For every maximum matching $M$, there exists an $M$-flower or an $M$-posy in $G$,
		\item For some maximum matching $M$, there exists an $M$-flower or an $M$-posy in $G$.
	\end{itemize}
\end{theorem}

Sterboul \cite{stersoul1979characterization} was the first to characterize
non–König--Egerváry graphs by means of forbidden configurations
relative to a maximum matching. Later, Korach, Nguyen, and Peis 
\cite{korach2006subgraph} reformulated this characterization 
in terms of simpler configurations, unifying flower- and posy-type structures. 
Subsequently, Bonomo et al. \cite{bonomo2013forbidden} obtained a purely 
structural characterization based on forbidden subgraphs. 
More recently, in \cite{jaume2025confpart1,jaume2025confpart2,jaume2025confpart3}, 
results were obtained that simplify the treatment of flower and posy structures.

Given a graph $G$ and an odd cycle $C$ of $G$, we define the \emph{reach set} of $C$ as
\[
R(C):=\U_{F}V(F),
\]
\noindent where the union is taken over all $M$-flowers $F$ of $G$
whose $M$-blossom is $C$, for some maximum matching $M$.

A graph $G$ is \emph{odd cycle disjoint} if every pair of distinct odd cycles 
$C$ and $C^{\P}$ in $G$ satisfies $V(C)\ii V(C^{\P})=\emptyset$. 
An $R$-\emph{disjoint graph} $G$ is a graph containing at least one odd cycle 
such that $R(C)\neq\emptyset$ and $R(C)\ii R(C^{\P})=\emptyset$
for every pair of distinct odd cycles $C$ and $C^{\P}$ of $G$. 
Hence, every $R$-\emph{disjoint graph} is also an odd cycle disjoint graph, 
so each odd cycle is chordless, and moreover $V(C)\s R(C)$.

In an $R$-disjoint graph $G$, one can naturally consider the following 
partition of the vertex set:
\[
\left\{ R(C):C\textnormal{ is an odd cycle of }G\right\}\u\{B(G)\},
\]
\noindent where $B(G)=V(G)-\U_{C}R(C)$. 
This is the \emph{flower decomposition} of $G$. 
It is immediate that $G[B(G)]$ is a bipartite graph, since it contains no odd cycles. 
Moreover, the only odd cycle in $G[R(C)]$ is $C$ itself.

Therefore, an $R$-disjoint graph with a unique odd cycle 
is an almost bipartite non–König--Egerváry graph \cite{kevin2025RDG}.

 \begin{theorem}
 	\cite{kevin2025RDG,levit2025almost} If $G$ is
 	an $R$-disjoint graph with exactly $k$ disjoint odd cycles,
 	then 
 	\begin{itemize}
 		\item $\ker(G)=\textnormal{core}(G),$
 		\item $\left|\textnormal{corona}(G)\right|+\left|\textnormal{core}(G)\right|=2\alpha(G)+k,$
 		\item $\textnormal{corona}(G)\u N(\textnormal{core}(G))=V(G).$
 	\end{itemize}
 \end{theorem}
 
 In \cref{sss1}, we provide a brief review of the notation that will be used throughout the paper. 
 In \cref{sss2}, we define and study structural properties of BAB-graphs, 
 a family that naturally generalizes $R$-disjoint graphs. Essentially, 
 a BAB-graph is constructed from a bipartite graph and several almost bipartite 
 non-König--Egerváry graphs, while preserving the structure induced by the 
 Gallai--Edmonds decomposition when connecting them. In this way, any result valid for 
 BAB-graphs also applies to almost bipartite non-König--Egerváry graphs 
 and to $R$-disjoint graphs. In this section, we obtain explicit formulas for 
 $\text{nucleus}(G)$ and $\text{diadem}(G)$, which allow us to generalize 
 some known results for $R$-disjoint graphs. 
 
 In \cref{sss3}, we prove that the determinant of the 
 adjacency matrix of a BAB-graph can be factorized in terms of the 
 determinants of the adjacency matrices of the graphs from which it is generated. 
 As a consequence, we confirm Conjecture 5.4 from \cite{kevin2025RDG}, which 
 establishes the validity of such factorization for $R$-disjoint graphs. 
 
 Finally, in \cref{sss4}, we analyze the structure of the independent sets in 
 BAB-graphs, derive bounds for 
 $\lvert\text{corona}(G)\rvert + \lvert\text{core}(G)\rvert$, provide an explicit 
 expression for $\text{ker}(G)$, and show examples of properties that are not preserved 
 when passing from $R$-disjoint graphs to BAB-graphs. In \cref{sss5}, we present several open problems.

\section{Preliminaries}\label{sss1}
All graphs considered in this paper are finite, undirected, and simple. 
For any undefined terminology or notation, we refer the reader to 
Lovász and Plummer \cite{LP} or Diestel \cite{Distel}.

Let \( G = (V, E) \) be a simple graph, where \( V = V(G) \) is the finite set of vertices and \( E = E(G) \) is the set of edges, with \( E \subseteq \{\{u, v\} : u, v \in V, u \neq v\} \). We denote the edge \( e=\{u, v\} \) as \( uv \). A subgraph of \( G \) is a graph \( H \) such that \( V(H) \subseteq V(G) \) and \( E(H) \subseteq E(G) \). A subgraph \( H \) of \( G \) is called a \textit{spanning} subgraph if \( V(H) = V(G) \). 

Let \( e \in E(G) \) and \( v \in V(G) \). We define \( G - e := (V, E \setminus \{e\}) \) and \( G - v := (V \setminus \{v\}, \{uw \in E : u,w \neq v\}) \). If \( X \subseteq V(G) \), the \textit{induced} subgraph of \( G \) by \( X \) is the subgraph \( G[X]=(X,F) \), where \( F:=\{uv \!\in\! E(G) : u, v \!\in \! X\} \).

Given a vertex set $S \subseteq V(G)$, we denote by $\partial(S)$ 
the set of edges having one endpoint in $S$ and the other in $V(G)-S$. 
We also denote by $\partial(S)$ the set of vertices in $S$ 
for which there exists an edge with one endpoint in $S$ and the other in $V(G)-S$.

A \textit{matching} \(M\) in a graph \(G\) is a set of pairwise non-adjacent edges. 
The \textit{matching number} of \(G\), denoted by  \(\mu(G)\), is the maximum cardinality of any matching in \(G\). 
Matchings induce an involution on the vertex set of the graph: \(M:V(G)\rightarrow V(G)\), where \(M(v)=u\) if \(uv \in M\), and \(M(v)=v\) otherwise. 
If \(S, U \subseteq V(G)\) with \(S \cap U = \emptyset\), we say that \(M\) is a matching from \(S\) to \(U\) if \(M(S) \subseteq U\). 

A matching $M$ is \emph{perfect} if $M(v)\neq v$ for every vertex 
of the graph. A graph $G$ is \emph{factor-critical} if $G-v$ has 
a perfect matching for every vertex $v$. 

A vertex set \( S \subseteq V \) is \textit{independent} if, for every pair of vertices \( u, v \in S \), we have \( uv \notin E \). 
The number of vertices in a maximum independent set is denoted by \( \alpha(G) \). 
A \textit{bipartite} graph is a graph whose vertex set can be partitioned into two disjoint independent sets. 
The number of vertices in a graph is called the \textit{order} of the graph.

A \textit{cycle} in $G$ is called \textit{odd} (resp. \textit{even}) if it has an odd (resp. even) number of edges. 
An \textit{even alternating cycle} with respect to a matching $M$ is an even cycle whose edges alternate between belonging and not belonging to $M$. 

Let $M$ be a matching in $G$. A path (or walk) is called \textit{alternating} with respect to $M$ if, for each pair of consecutive edges in the path, exactly one of them belongs to $M$. 
If the matching is clear from context, we simply say that the path is alternating. 
Given an alternating path (or walk), we say that $P$ is: an \textit{\(mm\)-\(M\)path} if it starts and ends with edges in $M$, 
an \textit{\(nn\)-\(M\)path} if it starts and ends with edges not in $M$, 
an \textit{\(mn\)-\(M\)path} if it starts with an edge in $M$ and ends with one not in $M$, 
and an \textit{\(nm\)-\(M\)path} if it starts with an unmatched edge and ends with a matched edge. 

For vertices $u,v \in V(G)$, we denote by $d_G(u,v)$ the length of a shortest path between $u$ and $v$ in $G$. 
We denote by $i(G)$ the number of isolated vertices of the graph $G$.

\section{Structural Properties of BAB-Graphs}\label{sss2}

In this section, we introduce the family of BAB-graphs. Essentially, 
a \emph{BAB-graph} is constructed from a bipartite graph and an 
almost bipartite non-König--Egerváry graph, preserving the structure 
induced by the Gallai--Edmonds decomposition when connecting them. 
This family naturally extends the class of $R$-disjoint graphs; 
in particular, every result obtained for BAB-graphs also applies to 
almost bipartite non-König--Egerváry graphs and $R$-disjoint graphs. 
Throughout this section, we study structural properties of BAB-graphs and, 
as a consequence, derive explicit formulas for $\text{nucleus}(G)$ 
and $\text{diadem}(G)$, thereby generalizing several known results. 

\begin{theorem}
	[\label{ge}\cite{edmonds1965paths,gallai1964maximale} Gallai--Edmonds
	structure theorem]
	Let $G$ be a graph, and define
	\begin{align*}
		D(G) & := \{ v : \textnormal{there exists a maximum matching that misses } v \}, \\
		A(G) & := \{ v : v \textnormal{ is adjacent to some } u \in D(G), \textnormal{ but } v \notin D(G) \}, \\
		C(G) & := V(G) - (D(G) \cup A(G)).
	\end{align*}
	
	\noindent If $G_{1}, \dots, G_{k}$ are the connected components of $G[D(G)]$
	and $M$ is a maximum matching of $G$, then:
	\begin{enumerate}
		\item $M$ covers $C(G)$ and matches $A(G)$ into distinct components of $G[D(G)]$.
		\item Each $G_i$ is a factor-critical graph, and the restriction of $M$ to $G_i$ 
		is a near-perfect matching.
	\end{enumerate}
\end{theorem}

The following lemma is a straightforward consequence of \cref{ge}, 
and will be useful later.

\begin{lemma}
	\label{1}
	Let $G$ and $G^{\P}$ be two graphs. Then the following statements hold:
	\begin{itemize}
		\item $D(G + e) = D(G)$ for every $e \subseteq A(G)$.
		\item $D(G + e) = D(G)$ for every $e = xy$ with $x \in A(G)$ and $y \in C(G)$.
		\item $D((G \cup G^{\P}) + e) = D(G) \cup D(G^{\P})$ for every $e = xy$
		with $x \in A(G)$ and $y \in A(G^{\P})$.
		\item $D((G \cup G^{\P}) + e) = D(G) \cup D(G^{\P})$ for every $e = xy$
		with $x \in A(G)$ and $y \in C(G^{\P})$.
	\end{itemize}
\end{lemma}

Let $B,G_{1},\dots,G_{k}$ be disjoint graphs where $B$ is a bipartite graph and each $G_{i}$ is an almost bipartite non-König--Egerváry graph with unique odd cycle $C(G_{i})$, for $i=1,\dots,k$. Moreover, for simplicity we shall assume that 
\[
R(C(G_{i}))=\U_{F}V(F)=V(G_{i}),
\]
\noindent for every $G_{i}$, that is, every vertex of $G_{i}$ belongs to a flower of $G_{i}$. Otherwise, it is easy to redefine the graphs by removing the vertices of $G_{i}$ that are not in any flower and adding them to $B$ so that this condition holds. Let $G$ be the graph obtained from the disjoint union $B\u G_{1}\u\dotsm\u G_{k}$ by adding edges between vertices belonging to different graphs, such that
\[
\partial_{G}(V(B))\s A(B)\u C(B)\textnormal{ and }\partial_{G}(V(G_{i}))\s A(G_{i}),
\]
\noindent for $i=1,\dots,k$. A graph constructed in this way is called a \emph{Bipartite--Almost Bipartite Connected Graph}, or briefly, a BAB-graph (see for instance \cref{Figura1}). We denote the BAB-graph $G$ by $(B,G_{1},\dots,G_{k})$. From \cref{1} we have the following observation:
\[
D(G)=D(B)\u D(G_{1})\u\dotsm\u D(G_{k}).
\]
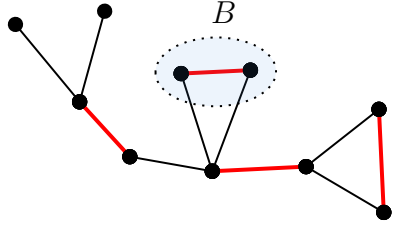
\begin{figure}[h!]
	
	\begin{center}

		\tikzset{every picture/.style={line width=0.75pt}} 
		
		\begin{tikzpicture}[x=0.75pt,y=0.75pt,yscale=-1,xscale=1]
			
			\draw    (198.2,137.55) -- (165.85,98.47) ;
			\draw [shift={(165.85,98.47)}, rotate = 230.38] [color={rgb, 255:red, 0; green, 0; blue, 0 }  ][fill={rgb, 255:red, 0; green, 0; blue, 0 }  ][line width=0.75]      (0, 0) circle [x radius= 3.35, y radius= 3.35]   ;
			\draw [shift={(198.2,137.55)}, rotate = 230.38] [color={rgb, 255:red, 0; green, 0; blue, 0 }  ][fill={rgb, 255:red, 0; green, 0; blue, 0 }  ][line width=0.75]      (0, 0) circle [x radius= 3.35, y radius= 3.35]   ;
			\draw    (210.83,91.8) -- (198.2,137.55) ;
			\draw [shift={(198.2,137.55)}, rotate = 105.43] [color={rgb, 255:red, 0; green, 0; blue, 0 }  ][fill={rgb, 255:red, 0; green, 0; blue, 0 }  ][line width=0.75]      (0, 0) circle [x radius= 3.35, y radius= 3.35]   ;
			\draw [shift={(210.83,91.8)}, rotate = 105.43] [color={rgb, 255:red, 0; green, 0; blue, 0 }  ][fill={rgb, 255:red, 0; green, 0; blue, 0 }  ][line width=0.75]      (0, 0) circle [x radius= 3.35, y radius= 3.35]   ;
			\draw    (223.53,165.25) -- (198.2,137.55) ;
			\draw [shift={(198.2,137.55)}, rotate = 227.55] [color={rgb, 255:red, 0; green, 0; blue, 0 }  ][fill={rgb, 255:red, 0; green, 0; blue, 0 }  ][line width=0.75]      (0, 0) circle [x radius= 3.35, y radius= 3.35]   ;
			\draw [shift={(223.53,165.25)}, rotate = 227.55] [color={rgb, 255:red, 0; green, 0; blue, 0 }  ][fill={rgb, 255:red, 0; green, 0; blue, 0 }  ][line width=0.75]      (0, 0) circle [x radius= 3.35, y radius= 3.35]   ;
			\draw    (265.1,172.58) -- (223.53,165.25) ;
			\draw [shift={(223.53,165.25)}, rotate = 190.01] [color={rgb, 255:red, 0; green, 0; blue, 0 }  ][fill={rgb, 255:red, 0; green, 0; blue, 0 }  ][line width=0.75]      (0, 0) circle [x radius= 3.35, y radius= 3.35]   ;
			\draw [shift={(265.1,172.58)}, rotate = 190.01] [color={rgb, 255:red, 0; green, 0; blue, 0 }  ][fill={rgb, 255:red, 0; green, 0; blue, 0 }  ][line width=0.75]      (0, 0) circle [x radius= 3.35, y radius= 3.35]   ;
			\draw    (312.39,170.25) -- (265.1,172.58) ;
			\draw [shift={(265.1,172.58)}, rotate = 177.17] [color={rgb, 255:red, 0; green, 0; blue, 0 }  ][fill={rgb, 255:red, 0; green, 0; blue, 0 }  ][line width=0.75]      (0, 0) circle [x radius= 3.35, y radius= 3.35]   ;
			\draw [shift={(312.39,170.25)}, rotate = 177.17] [color={rgb, 255:red, 0; green, 0; blue, 0 }  ][fill={rgb, 255:red, 0; green, 0; blue, 0 }  ][line width=0.75]      (0, 0) circle [x radius= 3.35, y radius= 3.35]   ;
			\draw    (349.17,141.31) -- (312.39,170.25) ;
			\draw [shift={(312.39,170.25)}, rotate = 141.8] [color={rgb, 255:red, 0; green, 0; blue, 0 }  ][fill={rgb, 255:red, 0; green, 0; blue, 0 }  ][line width=0.75]      (0, 0) circle [x radius= 3.35, y radius= 3.35]   ;
			\draw [shift={(349.17,141.31)}, rotate = 141.8] [color={rgb, 255:red, 0; green, 0; blue, 0 }  ][fill={rgb, 255:red, 0; green, 0; blue, 0 }  ][line width=0.75]      (0, 0) circle [x radius= 3.35, y radius= 3.35]   ;
			\draw    (351.6,193.35) -- (312.39,170.25) ;
			\draw [shift={(312.39,170.25)}, rotate = 210.51] [color={rgb, 255:red, 0; green, 0; blue, 0 }  ][fill={rgb, 255:red, 0; green, 0; blue, 0 }  ][line width=0.75]      (0, 0) circle [x radius= 3.35, y radius= 3.35]   ;
			\draw [shift={(351.6,193.35)}, rotate = 210.51] [color={rgb, 255:red, 0; green, 0; blue, 0 }  ][fill={rgb, 255:red, 0; green, 0; blue, 0 }  ][line width=0.75]      (0, 0) circle [x radius= 3.35, y radius= 3.35]   ;
			\draw    (349.17,141.31) -- (351.6,193.35) ;
			\draw [shift={(351.6,193.35)}, rotate = 87.33] [color={rgb, 255:red, 0; green, 0; blue, 0 }  ][fill={rgb, 255:red, 0; green, 0; blue, 0 }  ][line width=0.75]      (0, 0) circle [x radius= 3.35, y radius= 3.35]   ;
			\draw [shift={(349.17,141.31)}, rotate = 87.33] [color={rgb, 255:red, 0; green, 0; blue, 0 }  ][fill={rgb, 255:red, 0; green, 0; blue, 0 }  ][line width=0.75]      (0, 0) circle [x radius= 3.35, y radius= 3.35]   ;
			\draw    (249.4,123.28) -- (265.1,172.58) ;
			\draw [shift={(265.1,172.58)}, rotate = 72.34] [color={rgb, 255:red, 0; green, 0; blue, 0 }  ][fill={rgb, 255:red, 0; green, 0; blue, 0 }  ][line width=0.75]      (0, 0) circle [x radius= 3.35, y radius= 3.35]   ;
			\draw [shift={(249.4,123.28)}, rotate = 72.34] [color={rgb, 255:red, 0; green, 0; blue, 0 }  ][fill={rgb, 255:red, 0; green, 0; blue, 0 }  ][line width=0.75]      (0, 0) circle [x radius= 3.35, y radius= 3.35]   ;
			\draw    (284.36,121.56) -- (265.1,172.58) ;
			\draw [shift={(265.1,172.58)}, rotate = 110.68] [color={rgb, 255:red, 0; green, 0; blue, 0 }  ][fill={rgb, 255:red, 0; green, 0; blue, 0 }  ][line width=0.75]      (0, 0) circle [x radius= 3.35, y radius= 3.35]   ;
			\draw [shift={(284.36,121.56)}, rotate = 110.68] [color={rgb, 255:red, 0; green, 0; blue, 0 }  ][fill={rgb, 255:red, 0; green, 0; blue, 0 }  ][line width=0.75]      (0, 0) circle [x radius= 3.35, y radius= 3.35]   ;
			\draw    (284.36,121.56) -- (249.4,123.28) ;
			\draw [shift={(249.4,123.28)}, rotate = 177.17] [color={rgb, 255:red, 0; green, 0; blue, 0 }  ][fill={rgb, 255:red, 0; green, 0; blue, 0 }  ][line width=0.75]      (0, 0) circle [x radius= 3.35, y radius= 3.35]   ;
			\draw [shift={(284.36,121.56)}, rotate = 177.17] [color={rgb, 255:red, 0; green, 0; blue, 0 }  ][fill={rgb, 255:red, 0; green, 0; blue, 0 }  ][line width=0.75]      (0, 0) circle [x radius= 3.35, y radius= 3.35]   ;
			\draw [color={rgb, 255:red, 255; green, 0; blue, 0 }  ,draw opacity=1 ][line width=1.5]    (349.17,141.31) -- (351.6,193.35) ;
			\draw [color={rgb, 255:red, 255; green, 0; blue, 0 }  ,draw opacity=1 ][line width=1.5]    (312.39,170.25) -- (265.1,172.58) ;
			\draw [color={rgb, 255:red, 255; green, 0; blue, 0 }  ,draw opacity=1 ][line width=1.5]    (284.36,121.56) -- (249.4,123.28) ;
			\draw [color={rgb, 255:red, 255; green, 0; blue, 0 }  ,draw opacity=1 ][line width=1.5]    (223.53,165.25) -- (198.2,137.55) ;
			\draw    (198.2,137.55) ;
			\draw [shift={(198.2,137.55)}, rotate = 0] [color={rgb, 255:red, 0; green, 0; blue, 0 }  ][fill={rgb, 255:red, 0; green, 0; blue, 0 }  ][line width=0.75]      (0, 0) circle [x radius= 3.35, y radius= 3.35]   ;
			\draw [shift={(198.2,137.55)}, rotate = 0] [color={rgb, 255:red, 0; green, 0; blue, 0 }  ][fill={rgb, 255:red, 0; green, 0; blue, 0 }  ][line width=0.75]      (0, 0) circle [x radius= 3.35, y radius= 3.35]   ;
			\draw    (223.53,165.25) ;
			\draw [shift={(223.53,165.25)}, rotate = 0] [color={rgb, 255:red, 0; green, 0; blue, 0 }  ][fill={rgb, 255:red, 0; green, 0; blue, 0 }  ][line width=0.75]      (0, 0) circle [x radius= 3.35, y radius= 3.35]   ;
			\draw [shift={(223.53,165.25)}, rotate = 0] [color={rgb, 255:red, 0; green, 0; blue, 0 }  ][fill={rgb, 255:red, 0; green, 0; blue, 0 }  ][line width=0.75]      (0, 0) circle [x radius= 3.35, y radius= 3.35]   ;
			\draw    (349.17,141.31) ;
			\draw [shift={(349.17,141.31)}, rotate = 0] [color={rgb, 255:red, 0; green, 0; blue, 0 }  ][fill={rgb, 255:red, 0; green, 0; blue, 0 }  ][line width=0.75]      (0, 0) circle [x radius= 3.35, y radius= 3.35]   ;
			\draw [shift={(349.17,141.31)}, rotate = 0] [color={rgb, 255:red, 0; green, 0; blue, 0 }  ][fill={rgb, 255:red, 0; green, 0; blue, 0 }  ][line width=0.75]      (0, 0) circle [x radius= 3.35, y radius= 3.35]   ;
			\draw    (312.39,170.25) ;
			\draw [shift={(312.39,170.25)}, rotate = 0] [color={rgb, 255:red, 0; green, 0; blue, 0 }  ][fill={rgb, 255:red, 0; green, 0; blue, 0 }  ][line width=0.75]      (0, 0) circle [x radius= 3.35, y radius= 3.35]   ;
			\draw [shift={(312.39,170.25)}, rotate = 0] [color={rgb, 255:red, 0; green, 0; blue, 0 }  ][fill={rgb, 255:red, 0; green, 0; blue, 0 }  ][line width=0.75]      (0, 0) circle [x radius= 3.35, y radius= 3.35]   ;
			\draw    (284.36,121.56) ;
			\draw [shift={(284.36,121.56)}, rotate = 0] [color={rgb, 255:red, 0; green, 0; blue, 0 }  ][fill={rgb, 255:red, 0; green, 0; blue, 0 }  ][line width=0.75]      (0, 0) circle [x radius= 3.35, y radius= 3.35]   ;
			\draw [shift={(284.36,121.56)}, rotate = 0] [color={rgb, 255:red, 0; green, 0; blue, 0 }  ][fill={rgb, 255:red, 0; green, 0; blue, 0 }  ][line width=0.75]      (0, 0) circle [x radius= 3.35, y radius= 3.35]   ;
			\draw    (265.1,172.58) ;
			\draw [shift={(265.1,172.58)}, rotate = 0] [color={rgb, 255:red, 0; green, 0; blue, 0 }  ][fill={rgb, 255:red, 0; green, 0; blue, 0 }  ][line width=0.75]      (0, 0) circle [x radius= 3.35, y radius= 3.35]   ;
			\draw [shift={(265.1,172.58)}, rotate = 0] [color={rgb, 255:red, 0; green, 0; blue, 0 }  ][fill={rgb, 255:red, 0; green, 0; blue, 0 }  ][line width=0.75]      (0, 0) circle [x radius= 3.35, y radius= 3.35]   ;
			\draw    (249.4,123.28) ;
			\draw [shift={(249.4,123.28)}, rotate = 0] [color={rgb, 255:red, 0; green, 0; blue, 0 }  ][fill={rgb, 255:red, 0; green, 0; blue, 0 }  ][line width=0.75]      (0, 0) circle [x radius= 3.35, y radius= 3.35]   ;
			\draw [shift={(249.4,123.28)}, rotate = 0] [color={rgb, 255:red, 0; green, 0; blue, 0 }  ][fill={rgb, 255:red, 0; green, 0; blue, 0 }  ][line width=0.75]      (0, 0) circle [x radius= 3.35, y radius= 3.35]   ;
			\draw    (351.6,193.35) ;
			\draw [shift={(351.6,193.35)}, rotate = 0] [color={rgb, 255:red, 0; green, 0; blue, 0 }  ][fill={rgb, 255:red, 0; green, 0; blue, 0 }  ][line width=0.75]      (0, 0) circle [x radius= 3.35, y radius= 3.35]   ;
			\draw [shift={(351.6,193.35)}, rotate = 0] [color={rgb, 255:red, 0; green, 0; blue, 0 }  ][fill={rgb, 255:red, 0; green, 0; blue, 0 }  ][line width=0.75]      (0, 0) circle [x radius= 3.35, y radius= 3.35]   ;
			\draw  [fill={rgb, 255:red, 74; green, 144; blue, 226 }  ,fill opacity=0.1 ][dash pattern={on 0.84pt off 2.51pt}] (236.49,122.78) .. controls (236.37,113.19) and (249.88,105.26) .. (266.67,105.06) .. controls (283.45,104.86) and (297.15,112.47) .. (297.26,122.06) .. controls (297.37,131.65) and (283.86,139.58) .. (267.08,139.78) .. controls (250.3,139.98) and (236.6,132.37) .. (236.49,122.78) -- cycle ;
			
			\draw (263,85.4) node [anchor=north west][inner sep=0.75pt]    {$B$};

		\end{tikzpicture}

	\end{center}
	\caption{In this example, a BAB-graph that is not an $R$-disjoint graph is shown.
	}
	\label{Figura1}
	
\end{figure}
Every almost bipartite non-König--Egerváry graph is an $R$-disjoint graph \cite{kevin2025RDG}, and every $R$-disjoint graph is a BAB-graph; this is a consequence of \cref{1}. Therefore, any result on BAB-graphs is also valid for almost bipartite non-König--Egerváry graphs and $R$-disjoint graphs. In this work, we will prove results in the class of BAB-graphs, from which both known and new results in the aforementioned classes can be derived.

\begin{lemma}
	\label{2} \cite{kevin2025RDG} Let $G$ be an $R$-disjoint graph and $C$ an odd cycle of $G$. Then for every vertex $v\in V(G)$:
	\begin{enumerate}
		\item If $v\in\partial(B(G))$, then $v\in A(G)\u C(G).$
		\item If $v\in\partial(R(C))$, then $v\in A(G).$
		\item $G[R(C)]$ is an almost bipartite non-König--Egerváry graph, and every vertex of $R(C)$ belongs to a flower of $G[R(C)]$.
		\item $G[B(G)]$ is a bipartite graph.
		\item $D(G)=D(G[B(G)])\u\U_{C}D(G[R(C)])$, where the union is taken over all odd cycles of $G$. 
	\end{enumerate}
\end{lemma}

As a first application of BAB-graphs, we have a characterization of $R$-disjoint graphs. This is a straightforward application of \cref{ge}, \cref{2}, and \cref{1}.

\begin{theorem}
	\label{3} The graph $G=(B,G_{1},\dots,G_{k})$ is a BAB-graph with exactly $k$ odd cycles if and only if $G$ is an $R$-disjoint graph with exactly $k$ odd cycles.
\end{theorem} 

In the remainder of this section, we will find explicit formulas for $\text{nucleus}(G)$ and $\text{diadem}(G)$ in terms of the Gallai-Edmonds decomposition of the BAB-graph. As a consequence, we generalize some known results in almost bipartite non-König--Egerváry graphs and $R$-disjoint graphs.

The following lemma is a straightforward consequence of the ear decomposition theorem for factor-critical graphs by Lovász \cite{lovasz1972structure}. 

\begin{lemma}
	\label{4} \cite{kevin2025RDG} Let $G$ be an almost bipartite non-König--Egerváry graph with a unique odd cycle $C$. Then $C$ is a connected component in $G[D(G)]$, and the remaining components are isolated vertices.
\end{lemma}

The result of the previous lemma extends to BAB-graphs as a corollary of \cref{2} and \cref{4}. 

\begin{corollary}
	\label{5} Let $G=(B,G_{1},\dots,G_{k})$ be a BAB-graph. Then $C(G_{1}),\dots,C(G_{k})$ are connected components of $G[D(G)]$, and the remaining components are isolated vertices.
\end{corollary}

\begin{lemma}
	\label{6} Let $G$ be a graph and $I$ a critical independent set. Then the following hold:
	\begin{itemize}
		\item $N(I)\ii D(G)=\emptyset$.
		\item $I\ii A(G)=\emptyset.$
		\item There exists a critical independent set $J$ of $G$ such that $J\s D(G)\ii I.$ 
	\end{itemize}
\end{lemma}

\begin{proof}
	By Hall's Theorem, there is a matching from $N(I)$ into $I$, since if $S\s N(I)$, then\\
	\begin{align*}
		& \left|\left(I-\left(N(S)\ii I\right)\right)\u S\right|-\left|N\left(\left(I-\left(N(S)\ii I\right)\right)\u S\right)\right|\\
		= & \left|I\right|-\left|N(S)\ii I\right|+\left|S\right|-\left|N(S)\ii I\right|-\left(\left|N(I)\right|-\left|S\right|\right)\\
		= & -2\left|N(S)\ii I\right|+2\left|S\right|+d(G)\\
		\le & d(G),
	\end{align*}
	
	\noindent hence $\left|S\right|\le\left|N(S)\ii I\right|$. From this fact, it is easy to see that every maximum matching of $G$ matches all vertices of $I$ into $N(I)$. Therefore, $N(I)\ii D(G)=\emptyset$. If there exists a vertex $v\in I\ii A(G)$, then any maximum matching matches $v$ in $I$, so $I\ii D(G)\neq\emptyset$, a contradiction. 
	
	If there exists a nonempty set $\emptyset\neq S\s N(I)$ such that $\left|N(S)\ii I\right|=\left|S\right|$, then we define $I^{\P}=I-\left(N(S)\ii I\right)$. Now $I^{\P}$ is a critical independent set of $G$, since 
	\begin{align*}
		\left|I^{\P}\right|-\left|N(I^{\P})\right| & =\left|I\right|-\left|N(S)\ii I\right|-\left|N(I)-S\right|\\
		& =\left|I\right|-\left|N(S)\ii I\right|-\left|N(I)\right|+\left|S\right|\\
		& =\left|I\right|-\left|N(I)\right|.
	\end{align*}
	
	\noindent Thus, we may assume that $\left|N(S)\ii I\right|>\left|S\right|$ for every nonempty $\emptyset\neq S\s N(I)$. Then, by Hall's Theorem, there exists a matching from the vertices of $S$ into the vertices of $I-v$ for every $v\in V(G)$. This matching can be extended to a maximum matching in $G$, and hence $I\s D(G)$. 
\end{proof}

\begin{lemma}
	\label{7} Let $G$ be a BAB-graph, and let $X$ be the set of isolated vertices in $G[D(G)]$ and $Y=D(G)-X$. Let $U$, $W$ be a bipartition of $G[C(G)]$, and let $I$ be a critical independent set of $G$. Then the following hold:
	\begin{itemize}
		\item $I\ii Y=\emptyset$.
		\item $N(X)=A(G)$.
		\item $X$ is a critical independent set of $G$. 
		\item $X\u U$ and $X\u W$ are maximum critical independent sets of $G$.
	\end{itemize}
\end{lemma}

\begin{proof}
	Since $G[Y]$ has no trivial components, every vertex in $Y$ has a neighbor in $Y$. Therefore, a critical independent set $I$ cannot contain vertices from $Y$, because by \cref{5}, $N(I)\ii Y\s N(I)\ii D(G)=\emptyset$, that is, $I\ii Y=\emptyset$. 
	
	Denote $G=(B,G_{1},\dots,G_{k})$, then $C(G_{1}),\dots,C(G_{k})$ are the components of $G[Y]$. We will show that there exists a maximum matching of $G$ that leaves exactly one vertex unmatched in each cycle $C(G_{i})$, for $i=1,\dots,k$. Note that a matching is maximum in $G$ if and only if it is maximum when restricted to each graph $B,G_{1},\dots,G_{k}$. 
	
	Let $M$ be a maximum matching of $G$ and $M_{i}=M\ii E(G_{i})$, for $i=1,\dots,k$. Then $M_{i}$ is a maximum matching of $G_{i}$, and since $G_{i}$ has a unique cycle, by \cref{sterboul}, there exists an $M_{i}$-flower of $G_{i}$ with $M_{i}$-blossom $C(G_{i})$ and $M_{i}$-stem $P_{i}$. Then $M_{i}\triangle E(P_{i})$ is a maximum matching of $G_{i}$ that leaves exactly one vertex of $C(G_{i})$ unmatched (the base of the blossom). Finally,
	\[
	M^{\P}=\left(M\ii E(B)\right)\u\U_{i=1}^{k}M_{i}\triangle E(P_{i})
	\]
	\noindent is a maximum matching of $G$ that leaves exactly one vertex unmatched in each cycle $C(G_{1}),\dots,C(G_{k})$. 
	
	By \cref{6}, there exists a critical independent set $I$ such that $I\s D(G)$. Since $I\ii Y=\emptyset$, we have $I\s X$. By \cref{ge}, $M^{\P}$ matches the vertices of $A(G)$ to vertices in $X$, hence $N(X)=A(G)$. We now show that $X$ is a critical independent set:
	\begin{eqnarray*}
		\left|X\right|-\left|N(X)\right| & = & \left|v\in X:M(v)=v\right|+\left|v\in X:M(v)\neq v\right|-\left|A\right|\\
		& = & \left|v\in X:M(v)=v\right|+\left|A\right|-\left|A\right|\\
		& = & \left|v\in X:M(v)=v\right|\\
		& \ge & \left|v\in I:M(v)=v\right|\\
		& \ge & \left(\left|v\in I:M(v)=v\right|+\left|v\in I:M(v)\neq v\right|\right)-\left|N(I)\right|\\
		& = & \left|I\right|-\left|N(I)\right|\\
		& = & d(G).
	\end{eqnarray*}
	To conclude, we show that $X\u U$ and $X\u W$ are maximum critical independent sets of $G$. Since $N(X)=A(G)$, we have
	\begin{align*}
		\left|X\u U\right|-\left|N(X\u U)\right| & =\left(\left|X\right|+\left|U\right|\right)-\left(\left|A\right|+\left|N(X\u U)-A(G)\right|\right)\\
		& =\left|X\right|+\left|U\right|-\left|A\right|-\left|W\right|\\
		& =\left|X\right|-\left|A\right|.
	\end{align*}
	The same reasoning applies to $X\u W$. By \cref{6}, a critical independent set cannot contain vertices from $A(G)$, nor vertices from $Y$. Therefore, $X\u U$ and $X\u W$ are maximum critical independent sets of $G$.
\end{proof}

From \cref{7}, explicit forms for $\textnormal{diadem}(G)$ and $\textnormal{nucleus}(G)$ can be obtained. 

\begin{theorem}
	\label{8} Let $G$ be a BAB-graph. Then
	\[
	\textnormal{diadem}(G)=X\u C(G),
	\]
	\noindent where $X$ is the set defined in \cref{7}.
\end{theorem}

\begin{proof}
	Let $U$, $W$ be a bipartition of $G[C(G)]$. By \cref{7}, the sets $X$, $X\u U$, $X\u W$ are critical independent sets of $G$, hence $X\u C(G)\s\textnormal{diadem}(G)$. On the other hand, by \cref{6}, the remaining vertices (vertices in $A(G)$ or in any nontrivial component of $G[D(G)]$) do not belong to any critical independent set.
\end{proof}

\begin{lemma}
	\label{9} \cite{levit2012critical} For a graph $G$, if $I$ and $J$ are critical in $G$, then $I\u J$ and $I\ii J$ are critical as well.
\end{lemma}

\begin{theorem}
	\label{10} Let $G$ be a BAB-graph. Then
	\[
	\textnormal{nucleus}(G)=X,
	\]
	\noindent where $X$ is the set defined in \cref{7}.
\end{theorem}

\begin{proof}
	By \cref{7}, $\textnormal{nucleus}(G)\s\textnormal{diadem}(G)\s X\u C(G)$, but $\textnormal{nucleus}(G)\ii C(G)=\emptyset$, since $\textnormal{nucleus}(G)\s(X\u U)\ii(X\u W)=X$, where $U$, $W$ is a bipartition of $G[C(G)]$. Moreover, by \cref{9}, $I\u X$ is a critical independent set for every critical independent set $I$. That is, a maximum critical independent set of $G$ always contains $X$.
\end{proof} 

As a corollary of \cref{8} and \cref{10}, we have the following.

\begin{corollary}
	\label{11} Let $G$ be a BAB-graph. Then $\textnormal{nucleus}(G)\s\textnormal{diadem}(G)$.
\end{corollary} 

\begin{corollary}
	Let $G$ be a BAB-graph with $C(G)=\emptyset$. Then
	\[
	\textnormal{nucleus}(G)=\textnormal{diadem}(G).
	\]
\end{corollary}

\begin{corollary}
	Let $G=(\emptyset,G_{1},\dots,G_{k})$ be a BAB-graph. Then
	\[
	\textnormal{nucleus}(G)=\textnormal{diadem}(G).
	\]
\end{corollary}

\begin{theorem}\label{ll2}
	Let $G$ be a BAB-graph. Then
	\[
	N[\textnormal{diadem}(G)]\u Y=V(G)\textnormal{ and }N[\textnormal{diadem}(G)]\ii Y=\emptyset,
	\]
	\noindent where $Y$ is the set defined in \cref{7}.
\end{theorem}

\begin{corollary}
	Let $G$ be a BAB-graph. Then
	\[
	\textnormal{nucleus}(G),D(G)-\textnormal{nucleus}(G),A(G),C(G)
	\]
	\noindent form a partition of $V(G)$.
\end{corollary}

By \cref{ll2}, $N[\textnormal{diadem}(G)]$ and $Y$ form a partition of $V(G)$. As a consequence, we have the following known result.

\begin{theorem}
	\cite{levit2024almost} If $G$ is an almost bipartite non-König--Egerváry graph with a unique odd cycle $C$, then
	\[
	V(C)\u N[\textnormal{diadem}(G)]=V(G).
	\]
\end{theorem}

\section{Determinantal Factorization in BAB-Graphs}\label{sss3}
In this section, we show that the determinant of the adjacency matrix of a BAB-graph can be factorized in terms of the determinants of the adjacency matrices of the graphs generating it. As a consequence, we confirm Conjecture 5.4 of \cite{kevin2025RDG}, which establishes the validity of this factorization for $R$-disjoint graphs. Finally, we explore some combinatorial implications derived from these results. To do so, we begin by introducing some preliminary results.

A spanning subgraph $S$ of $G$ is called a \emph{Sachs subgraph} of $G$ if every component of $S$ is either $K_2$ or a cycle. Let $\Sachs{G}$ denote the set of all Sachs subgraphs of $G$. If $G$ has a perfect matching, then for every $M \in \mathcal{M}(G)$ we have $M \in \Sachs{G}$. In particular, if $G$ has a perfect matching, then $|\Sachs{G}| \geq 1$ \cite{s3}.

The set of all bijective functions of a set $A$, called the permutations of $A$, is denoted by $S_{A}$. For more details about permutations, see \cite{rotman2012introduction}. Let $G$ be a graph with vertex set $V$. The elements of $S_{V}$ induce all the Sachs subgraphs of $G$ as follows: given a graph $G$, let $\sigma$ be a permutation of $V$. Then $\sigma(G)=(V,E_{\sigma}(G))$, where $E_{\sigma}(G)=\{uv\in E(G):\sigma(u)=v\}$, is a Sachs subgraph of $G$. A function $f:S_{V}\to\R$ is said to be $G$-Sachs-stable if $f(\sigma)=f(\tau)$ whenever $\sigma(G)=\tau(G)$, i.e., the permutations of $V$ induce the same Sachs subgraph of $G$. For example, the sign of a permutation is $G$-Sachs-stable for any graph $G$. A $G$-Sachs-stable function induces a function on the Sachs subgraphs of a graph. If $f:S_{V}\to\R$ is $G$-Sachs-stable and $S\in\text{Sachs}(G)$, then $f(S):=f(\sigma)$ for any permutation $\sigma$ such that $\sigma(G)=S$.

If $G$ is a graph with vertices $[n]:=\{1,\dots,n\}$ and $f:S_{[n]}\to\R$ is a $G$-Sachs-stable function, the $f$-Schur function of $G$ is defined by
\[
\text{schur}_{f}(G):=\underset{\sigma\in S_{[n]}}{\sm}f(\sigma)\prod\limits_{i=1}^{n}a_{i\sigma(i)}=\underset{S\in\text{Sachs}(G)}{\sm}f(S)2^{\left|C(S)\right|},
\]
\noindent where $A(G)=[a_{ij}]_{1\le i,j\le n}$ is the adjacency matrix of $G$, and $|C(S)|$ is the number of cycles of $S$, see \cite{xxminc1984permanents,aaas}.  

\begin{theorem}
	[\cite{h2}\label{harary}] Let $G$ be a graph. Then
	\[
	\textnormal{det}(G)=\underset{S\in\textnormal{Sachs}(G)}{\sm} (-1)^k 2^{\left|C(S)\right|},
	\]
	\noindent where $k$ is the number of even cycles of $S$.
\end{theorem}

We denote by $\det(G)$ the determinant $\det(A(G))$ and define $\det(G[\emptyset])=1$.

\begin{theorem}
	\label{12} \cite{tutte19531} A graph $G$ has a Sachs subgraph if and only if
	\[
	i(G-S)\leq\left|S\right|,
	\]
	\noindent for every $S\s V(G)$.
\end{theorem}

\begin{lemma}
	\label{13} Let $G=(B,G_{1},\dots,G_{k})$ be a BAB-graph with $\textnormal{Sachs}(G)\neq\emptyset$. Then $\left|\textnormal{nucleus}(G)\right|=\left|A(G)\right|$.
\end{lemma}

\begin{proof}
	Let $I$ be a critical independent set of $G$. Since $G$ has a Sachs subgraph, by \cref{12} 
	\[
	\left|I\right|\le i(G-N(I))\le\left|N(I)\right|.
	\]
	\noindent That is, $d(G)=0$. On the other hand, by \cref{7}, $X$ is a critical independent set of $G$ and $N(\textnormal{nucleus}(G))=A(G)$, hence
	\[
	d(G)=0=\left|\textnormal{nucleus}(G)\right|-\left|A(G)\right|.
	\]
	\noindent As we wanted to show.
\end{proof}

The following theorem shows the structure of Sachs subgraphs in a BAB-graph. As a consequence, the Schur functions of $G$ admit a factorization in terms of the subgraphs that compose it.

\begin{theorem}
	\label{14} Let $G=(B,G_{1},\dots,G_{k})$ be a BAB-graph and $H\in\textnormal{Sachs}(G)$. If $H^{\P}$ is a connected component of $H$, then either $V(H^{\P})\s V(B)$ or $V(H^{\P})\s V(C(G_{i}))$ for some $i=1,\dots,k$.
\end{theorem} 

\begin{proof}
	Let $H\in\textnormal{Sachs}(G)$ and, by contradiction, suppose there exists an edge $xy=e\in E(H)$ such that $x\in A(G)$ and $y\in A(G)\u C(G)$. If $e$ is in an even component of $H$, then $G^{\p}=G-x-y$ has a Sachs subgraph. On the other hand,
	\[
	\left|\textnormal{nucleus}(G)\right|\le i(G^{\P}-\left(A(G)-\{x\}\right))\le\left|A(G)\right|-1.
	\]
	\noindent This is absurd by \cref{13}. If $e$ is in an odd component of $H$, then $G^{\P}=G-x$ has a Sachs subgraph, and similarly 
	\[
	\left|\textnormal{nucleus}(G)\right|\le i(G^{\P}-\left(A(G)-\{x\}\right))\le\left|A(G)\right|-1,
	\] 
	absurd.
\end{proof}

As a corollary of \cref{14} and \cref{harary}, we obtain the determinant factorization theorem for BAB-graphs (see \cref{Figura2}).

\begin{theorem}
	\label{x14} Let $G=(B,G_{1},\dots,G_{k})$ be a BAB-graph. Then
	\[
	\det(G)=\det(B)\prod_{i=1}^{k}\det G_{i}.
	\]
\end{theorem}

Hence, as a particular case for $R$-disjoint graphs, \cref{x14} verifies the following conjecture (Conjecture 5.4 in \cite{kevin2025RDG}).

\begin{conjecture}
	[\cite{kevin2025RDG}] Let $G$ be an $R$-disjoint graph. Then
	\[
	\det A(G)=\det A(G[B(G)])\prod_{C}\det A(G[R(C)]).
	\]
	
	\noindent Where the product is taken over all odd cycles $C$ of $G$.
\end{conjecture}

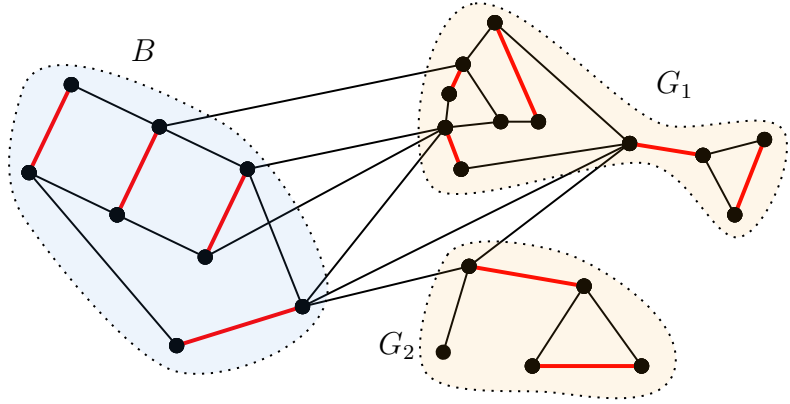
\begin{figure}[h!]
	
	\begin{center}

		\tikzset{every picture/.style={line width=0.75pt}} 
		
		\begin{tikzpicture}[x=0.75pt,y=0.75pt,yscale=-1,xscale=1]
			
			\draw    (133.64,134.35) -- (89.22,113.03) ;
			\draw [shift={(89.22,113.03)}, rotate = 205.64] [color={rgb, 255:red, 0; green, 0; blue, 0 }  ][fill={rgb, 255:red, 0; green, 0; blue, 0 }  ][line width=0.75]      (0, 0) circle [x radius= 3.35, y radius= 3.35]   ;
			\draw [shift={(133.64,134.35)}, rotate = 205.64] [color={rgb, 255:red, 0; green, 0; blue, 0 }  ][fill={rgb, 255:red, 0; green, 0; blue, 0 }  ][line width=0.75]      (0, 0) circle [x radius= 3.35, y radius= 3.35]   ;
			\draw    (89.22,113.03) -- (110.54,68.61) ;
			\draw [shift={(110.54,68.61)}, rotate = 295.64] [color={rgb, 255:red, 0; green, 0; blue, 0 }  ][fill={rgb, 255:red, 0; green, 0; blue, 0 }  ][line width=0.75]      (0, 0) circle [x radius= 3.35, y radius= 3.35]   ;
			\draw [shift={(89.22,113.03)}, rotate = 295.64] [color={rgb, 255:red, 0; green, 0; blue, 0 }  ][fill={rgb, 255:red, 0; green, 0; blue, 0 }  ][line width=0.75]      (0, 0) circle [x radius= 3.35, y radius= 3.35]   ;
			\draw    (133.64,134.35) -- (154.96,89.93) ;
			\draw [shift={(154.96,89.93)}, rotate = 295.64] [color={rgb, 255:red, 0; green, 0; blue, 0 }  ][fill={rgb, 255:red, 0; green, 0; blue, 0 }  ][line width=0.75]      (0, 0) circle [x radius= 3.35, y radius= 3.35]   ;
			\draw [shift={(133.64,134.35)}, rotate = 295.64] [color={rgb, 255:red, 0; green, 0; blue, 0 }  ][fill={rgb, 255:red, 0; green, 0; blue, 0 }  ][line width=0.75]      (0, 0) circle [x radius= 3.35, y radius= 3.35]   ;
			\draw    (178.05,155.67) -- (133.64,134.35) ;
			\draw [shift={(133.64,134.35)}, rotate = 205.64] [color={rgb, 255:red, 0; green, 0; blue, 0 }  ][fill={rgb, 255:red, 0; green, 0; blue, 0 }  ][line width=0.75]      (0, 0) circle [x radius= 3.35, y radius= 3.35]   ;
			\draw [shift={(178.05,155.67)}, rotate = 205.64] [color={rgb, 255:red, 0; green, 0; blue, 0 }  ][fill={rgb, 255:red, 0; green, 0; blue, 0 }  ][line width=0.75]      (0, 0) circle [x radius= 3.35, y radius= 3.35]   ;
			\draw    (199.37,111.25) -- (154.96,89.93) ;
			\draw [shift={(154.96,89.93)}, rotate = 205.64] [color={rgb, 255:red, 0; green, 0; blue, 0 }  ][fill={rgb, 255:red, 0; green, 0; blue, 0 }  ][line width=0.75]      (0, 0) circle [x radius= 3.35, y radius= 3.35]   ;
			\draw [shift={(199.37,111.25)}, rotate = 205.64] [color={rgb, 255:red, 0; green, 0; blue, 0 }  ][fill={rgb, 255:red, 0; green, 0; blue, 0 }  ][line width=0.75]      (0, 0) circle [x radius= 3.35, y radius= 3.35]   ;
			\draw    (154.96,89.93) -- (110.54,68.61) ;
			\draw [shift={(110.54,68.61)}, rotate = 205.64] [color={rgb, 255:red, 0; green, 0; blue, 0 }  ][fill={rgb, 255:red, 0; green, 0; blue, 0 }  ][line width=0.75]      (0, 0) circle [x radius= 3.35, y radius= 3.35]   ;
			\draw [shift={(154.96,89.93)}, rotate = 205.64] [color={rgb, 255:red, 0; green, 0; blue, 0 }  ][fill={rgb, 255:red, 0; green, 0; blue, 0 }  ][line width=0.75]      (0, 0) circle [x radius= 3.35, y radius= 3.35]   ;
			\draw    (178.05,155.67) -- (199.37,111.25) ;
			\draw [shift={(199.37,111.25)}, rotate = 295.64] [color={rgb, 255:red, 0; green, 0; blue, 0 }  ][fill={rgb, 255:red, 0; green, 0; blue, 0 }  ][line width=0.75]      (0, 0) circle [x radius= 3.35, y radius= 3.35]   ;
			\draw [shift={(178.05,155.67)}, rotate = 295.64] [color={rgb, 255:red, 0; green, 0; blue, 0 }  ][fill={rgb, 255:red, 0; green, 0; blue, 0 }  ][line width=0.75]      (0, 0) circle [x radius= 3.35, y radius= 3.35]   ;
			\draw    (227.13,180.67) -- (163.64,200.35) ;
			\draw [shift={(163.64,200.35)}, rotate = 162.78] [color={rgb, 255:red, 0; green, 0; blue, 0 }  ][fill={rgb, 255:red, 0; green, 0; blue, 0 }  ][line width=0.75]      (0, 0) circle [x radius= 3.35, y radius= 3.35]   ;
			\draw [shift={(227.13,180.67)}, rotate = 162.78] [color={rgb, 255:red, 0; green, 0; blue, 0 }  ][fill={rgb, 255:red, 0; green, 0; blue, 0 }  ][line width=0.75]      (0, 0) circle [x radius= 3.35, y radius= 3.35]   ;
			\draw    (163.64,200.35) -- (89.22,113.03) ;
			\draw [shift={(89.22,113.03)}, rotate = 229.56] [color={rgb, 255:red, 0; green, 0; blue, 0 }  ][fill={rgb, 255:red, 0; green, 0; blue, 0 }  ][line width=0.75]      (0, 0) circle [x radius= 3.35, y radius= 3.35]   ;
			\draw [shift={(163.64,200.35)}, rotate = 229.56] [color={rgb, 255:red, 0; green, 0; blue, 0 }  ][fill={rgb, 255:red, 0; green, 0; blue, 0 }  ][line width=0.75]      (0, 0) circle [x radius= 3.35, y radius= 3.35]   ;
			\draw    (199.37,111.25) -- (227.13,180.67) ;
			\draw [shift={(227.13,180.67)}, rotate = 68.2] [color={rgb, 255:red, 0; green, 0; blue, 0 }  ][fill={rgb, 255:red, 0; green, 0; blue, 0 }  ][line width=0.75]      (0, 0) circle [x radius= 3.35, y radius= 3.35]   ;
			\draw [shift={(199.37,111.25)}, rotate = 68.2] [color={rgb, 255:red, 0; green, 0; blue, 0 }  ][fill={rgb, 255:red, 0; green, 0; blue, 0 }  ][line width=0.75]      (0, 0) circle [x radius= 3.35, y radius= 3.35]   ;
			\draw    (343.05,210.67) -- (398.13,210.67) ;
			\draw [shift={(398.13,210.67)}, rotate = 0] [color={rgb, 255:red, 0; green, 0; blue, 0 }  ][fill={rgb, 255:red, 0; green, 0; blue, 0 }  ][line width=0.75]      (0, 0) circle [x radius= 3.35, y radius= 3.35]   ;
			\draw [shift={(343.05,210.67)}, rotate = 0] [color={rgb, 255:red, 0; green, 0; blue, 0 }  ][fill={rgb, 255:red, 0; green, 0; blue, 0 }  ][line width=0.75]      (0, 0) circle [x radius= 3.35, y radius= 3.35]   ;
			\draw    (343.05,210.67) -- (369.13,170.33) ;
			\draw [shift={(369.13,170.33)}, rotate = 302.89] [color={rgb, 255:red, 0; green, 0; blue, 0 }  ][fill={rgb, 255:red, 0; green, 0; blue, 0 }  ][line width=0.75]      (0, 0) circle [x radius= 3.35, y radius= 3.35]   ;
			\draw [shift={(343.05,210.67)}, rotate = 302.89] [color={rgb, 255:red, 0; green, 0; blue, 0 }  ][fill={rgb, 255:red, 0; green, 0; blue, 0 }  ][line width=0.75]      (0, 0) circle [x radius= 3.35, y radius= 3.35]   ;
			\draw    (369.13,170.33) -- (398.13,210.67) ;
			\draw [shift={(398.13,210.67)}, rotate = 54.28] [color={rgb, 255:red, 0; green, 0; blue, 0 }  ][fill={rgb, 255:red, 0; green, 0; blue, 0 }  ][line width=0.75]      (0, 0) circle [x radius= 3.35, y radius= 3.35]   ;
			\draw [shift={(369.13,170.33)}, rotate = 54.28] [color={rgb, 255:red, 0; green, 0; blue, 0 }  ][fill={rgb, 255:red, 0; green, 0; blue, 0 }  ][line width=0.75]      (0, 0) circle [x radius= 3.35, y radius= 3.35]   ;
			\draw    (311.13,160.33) -- (369.13,170.33) ;
			\draw [shift={(369.13,170.33)}, rotate = 9.78] [color={rgb, 255:red, 0; green, 0; blue, 0 }  ][fill={rgb, 255:red, 0; green, 0; blue, 0 }  ][line width=0.75]      (0, 0) circle [x radius= 3.35, y radius= 3.35]   ;
			\draw [shift={(311.13,160.33)}, rotate = 9.78] [color={rgb, 255:red, 0; green, 0; blue, 0 }  ][fill={rgb, 255:red, 0; green, 0; blue, 0 }  ][line width=0.75]      (0, 0) circle [x radius= 3.35, y radius= 3.35]   ;
			\draw    (311.13,160.33) -- (298.13,203.67) ;
			\draw [shift={(298.13,203.67)}, rotate = 106.7] [color={rgb, 255:red, 0; green, 0; blue, 0 }  ][fill={rgb, 255:red, 0; green, 0; blue, 0 }  ][line width=0.75]      (0, 0) circle [x radius= 3.35, y radius= 3.35]   ;
			\draw [shift={(311.13,160.33)}, rotate = 106.7] [color={rgb, 255:red, 0; green, 0; blue, 0 }  ][fill={rgb, 255:red, 0; green, 0; blue, 0 }  ][line width=0.75]      (0, 0) circle [x radius= 3.35, y radius= 3.35]   ;
			\draw [color={rgb, 255:red, 255; green, 0; blue, 0 }  ,draw opacity=1 ][line width=1.5]    (398.13,210.67) -- (343.05,210.67) ;
			\draw [color={rgb, 255:red, 255; green, 0; blue, 0 }  ,draw opacity=1 ][line width=1.5]    (311.13,160.33) -- (369.13,170.33) ;
			\draw    (445.13,134.33) -- (460.13,96.33) ;
			\draw [shift={(460.13,96.33)}, rotate = 291.54] [color={rgb, 255:red, 0; green, 0; blue, 0 }  ][fill={rgb, 255:red, 0; green, 0; blue, 0 }  ][line width=0.75]      (0, 0) circle [x radius= 3.35, y radius= 3.35]   ;
			\draw [shift={(445.13,134.33)}, rotate = 291.54] [color={rgb, 255:red, 0; green, 0; blue, 0 }  ][fill={rgb, 255:red, 0; green, 0; blue, 0 }  ][line width=0.75]      (0, 0) circle [x radius= 3.35, y radius= 3.35]   ;
			\draw    (445.13,134.33) -- (429.13,104.33) ;
			\draw [shift={(429.13,104.33)}, rotate = 241.93] [color={rgb, 255:red, 0; green, 0; blue, 0 }  ][fill={rgb, 255:red, 0; green, 0; blue, 0 }  ][line width=0.75]      (0, 0) circle [x radius= 3.35, y radius= 3.35]   ;
			\draw [shift={(445.13,134.33)}, rotate = 241.93] [color={rgb, 255:red, 0; green, 0; blue, 0 }  ][fill={rgb, 255:red, 0; green, 0; blue, 0 }  ][line width=0.75]      (0, 0) circle [x radius= 3.35, y radius= 3.35]   ;
			\draw    (429.13,104.33) -- (460.13,96.33) ;
			\draw [shift={(460.13,96.33)}, rotate = 345.53] [color={rgb, 255:red, 0; green, 0; blue, 0 }  ][fill={rgb, 255:red, 0; green, 0; blue, 0 }  ][line width=0.75]      (0, 0) circle [x radius= 3.35, y radius= 3.35]   ;
			\draw [shift={(429.13,104.33)}, rotate = 345.53] [color={rgb, 255:red, 0; green, 0; blue, 0 }  ][fill={rgb, 255:red, 0; green, 0; blue, 0 }  ][line width=0.75]      (0, 0) circle [x radius= 3.35, y radius= 3.35]   ;
			\draw    (392.13,98.33) -- (429.13,104.33) ;
			\draw [shift={(429.13,104.33)}, rotate = 9.21] [color={rgb, 255:red, 0; green, 0; blue, 0 }  ][fill={rgb, 255:red, 0; green, 0; blue, 0 }  ][line width=0.75]      (0, 0) circle [x radius= 3.35, y radius= 3.35]   ;
			\draw [shift={(392.13,98.33)}, rotate = 9.21] [color={rgb, 255:red, 0; green, 0; blue, 0 }  ][fill={rgb, 255:red, 0; green, 0; blue, 0 }  ][line width=0.75]      (0, 0) circle [x radius= 3.35, y radius= 3.35]   ;
			\draw    (392.13,98.33) -- (324.13,37.33) ;
			\draw [shift={(324.13,37.33)}, rotate = 221.89] [color={rgb, 255:red, 0; green, 0; blue, 0 }  ][fill={rgb, 255:red, 0; green, 0; blue, 0 }  ][line width=0.75]      (0, 0) circle [x radius= 3.35, y radius= 3.35]   ;
			\draw [shift={(392.13,98.33)}, rotate = 221.89] [color={rgb, 255:red, 0; green, 0; blue, 0 }  ][fill={rgb, 255:red, 0; green, 0; blue, 0 }  ][line width=0.75]      (0, 0) circle [x radius= 3.35, y radius= 3.35]   ;
			\draw    (392.13,98.33) -- (307.13,111.33) ;
			\draw [shift={(307.13,111.33)}, rotate = 171.3] [color={rgb, 255:red, 0; green, 0; blue, 0 }  ][fill={rgb, 255:red, 0; green, 0; blue, 0 }  ][line width=0.75]      (0, 0) circle [x radius= 3.35, y radius= 3.35]   ;
			\draw [shift={(392.13,98.33)}, rotate = 171.3] [color={rgb, 255:red, 0; green, 0; blue, 0 }  ][fill={rgb, 255:red, 0; green, 0; blue, 0 }  ][line width=0.75]      (0, 0) circle [x radius= 3.35, y radius= 3.35]   ;
			\draw    (299.13,90.33) -- (307.13,111.33) ;
			\draw [shift={(307.13,111.33)}, rotate = 69.15] [color={rgb, 255:red, 0; green, 0; blue, 0 }  ][fill={rgb, 255:red, 0; green, 0; blue, 0 }  ][line width=0.75]      (0, 0) circle [x radius= 3.35, y radius= 3.35]   ;
			\draw [shift={(299.13,90.33)}, rotate = 69.15] [color={rgb, 255:red, 0; green, 0; blue, 0 }  ][fill={rgb, 255:red, 0; green, 0; blue, 0 }  ][line width=0.75]      (0, 0) circle [x radius= 3.35, y radius= 3.35]   ;
			\draw    (299.13,90.33) -- (301.13,73.33) ;
			\draw [shift={(301.13,73.33)}, rotate = 276.71] [color={rgb, 255:red, 0; green, 0; blue, 0 }  ][fill={rgb, 255:red, 0; green, 0; blue, 0 }  ][line width=0.75]      (0, 0) circle [x radius= 3.35, y radius= 3.35]   ;
			\draw [shift={(299.13,90.33)}, rotate = 276.71] [color={rgb, 255:red, 0; green, 0; blue, 0 }  ][fill={rgb, 255:red, 0; green, 0; blue, 0 }  ][line width=0.75]      (0, 0) circle [x radius= 3.35, y radius= 3.35]   ;
			\draw    (324.13,37.33) -- (308.13,58.33) ;
			\draw [shift={(308.13,58.33)}, rotate = 127.3] [color={rgb, 255:red, 0; green, 0; blue, 0 }  ][fill={rgb, 255:red, 0; green, 0; blue, 0 }  ][line width=0.75]      (0, 0) circle [x radius= 3.35, y radius= 3.35]   ;
			\draw [shift={(324.13,37.33)}, rotate = 127.3] [color={rgb, 255:red, 0; green, 0; blue, 0 }  ][fill={rgb, 255:red, 0; green, 0; blue, 0 }  ][line width=0.75]      (0, 0) circle [x radius= 3.35, y radius= 3.35]   ;
			\draw    (301.13,73.33) -- (308.13,58.33) ;
			\draw [shift={(308.13,58.33)}, rotate = 295.02] [color={rgb, 255:red, 0; green, 0; blue, 0 }  ][fill={rgb, 255:red, 0; green, 0; blue, 0 }  ][line width=0.75]      (0, 0) circle [x radius= 3.35, y radius= 3.35]   ;
			\draw [shift={(301.13,73.33)}, rotate = 295.02] [color={rgb, 255:red, 0; green, 0; blue, 0 }  ][fill={rgb, 255:red, 0; green, 0; blue, 0 }  ][line width=0.75]      (0, 0) circle [x radius= 3.35, y radius= 3.35]   ;
			\draw    (299.13,90.33) -- (327.13,87.33) ;
			\draw [shift={(327.13,87.33)}, rotate = 353.88] [color={rgb, 255:red, 0; green, 0; blue, 0 }  ][fill={rgb, 255:red, 0; green, 0; blue, 0 }  ][line width=0.75]      (0, 0) circle [x radius= 3.35, y radius= 3.35]   ;
			\draw [shift={(299.13,90.33)}, rotate = 353.88] [color={rgb, 255:red, 0; green, 0; blue, 0 }  ][fill={rgb, 255:red, 0; green, 0; blue, 0 }  ][line width=0.75]      (0, 0) circle [x radius= 3.35, y radius= 3.35]   ;
			\draw    (308.13,58.33) -- (327.13,87.33) ;
			\draw [shift={(327.13,87.33)}, rotate = 56.77] [color={rgb, 255:red, 0; green, 0; blue, 0 }  ][fill={rgb, 255:red, 0; green, 0; blue, 0 }  ][line width=0.75]      (0, 0) circle [x radius= 3.35, y radius= 3.35]   ;
			\draw [shift={(308.13,58.33)}, rotate = 56.77] [color={rgb, 255:red, 0; green, 0; blue, 0 }  ][fill={rgb, 255:red, 0; green, 0; blue, 0 }  ][line width=0.75]      (0, 0) circle [x radius= 3.35, y radius= 3.35]   ;
			\draw    (327.13,87.33) -- (346.13,87.33) ;
			\draw [shift={(346.13,87.33)}, rotate = 0] [color={rgb, 255:red, 0; green, 0; blue, 0 }  ][fill={rgb, 255:red, 0; green, 0; blue, 0 }  ][line width=0.75]      (0, 0) circle [x radius= 3.35, y radius= 3.35]   ;
			\draw [shift={(327.13,87.33)}, rotate = 0] [color={rgb, 255:red, 0; green, 0; blue, 0 }  ][fill={rgb, 255:red, 0; green, 0; blue, 0 }  ][line width=0.75]      (0, 0) circle [x radius= 3.35, y radius= 3.35]   ;
			\draw    (346.13,87.33) -- (324.13,37.33) ;
			\draw [shift={(324.13,37.33)}, rotate = 246.25] [color={rgb, 255:red, 0; green, 0; blue, 0 }  ][fill={rgb, 255:red, 0; green, 0; blue, 0 }  ][line width=0.75]      (0, 0) circle [x radius= 3.35, y radius= 3.35]   ;
			\draw [shift={(346.13,87.33)}, rotate = 246.25] [color={rgb, 255:red, 0; green, 0; blue, 0 }  ][fill={rgb, 255:red, 0; green, 0; blue, 0 }  ][line width=0.75]      (0, 0) circle [x radius= 3.35, y radius= 3.35]   ;
			\draw [color={rgb, 255:red, 255; green, 0; blue, 0 }  ,draw opacity=1 ][line width=1.5]    (163.64,200.35) -- (227.13,180.67) ;
			\draw [color={rgb, 255:red, 255; green, 0; blue, 0 }  ,draw opacity=1 ][line width=1.5]    (89.22,113.03) -- (110.54,68.61) ;
			\draw [color={rgb, 255:red, 255; green, 0; blue, 0 }  ,draw opacity=1 ][line width=1.5]    (133.64,134.35) -- (154.96,89.93) ;
			\draw [color={rgb, 255:red, 255; green, 0; blue, 0 }  ,draw opacity=1 ][line width=1.5]    (178.05,155.67) -- (199.37,111.25) ;
			\draw    (311.13,160.33) -- (227.13,180.67) ;
			\draw [shift={(227.13,180.67)}, rotate = 166.39] [color={rgb, 255:red, 0; green, 0; blue, 0 }  ][fill={rgb, 255:red, 0; green, 0; blue, 0 }  ][line width=0.75]      (0, 0) circle [x radius= 3.35, y radius= 3.35]   ;
			\draw [shift={(311.13,160.33)}, rotate = 166.39] [color={rgb, 255:red, 0; green, 0; blue, 0 }  ][fill={rgb, 255:red, 0; green, 0; blue, 0 }  ][line width=0.75]      (0, 0) circle [x radius= 3.35, y radius= 3.35]   ;
			\draw    (392.13,98.33) -- (311.13,160.33) ;
			\draw [shift={(311.13,160.33)}, rotate = 142.57] [color={rgb, 255:red, 0; green, 0; blue, 0 }  ][fill={rgb, 255:red, 0; green, 0; blue, 0 }  ][line width=0.75]      (0, 0) circle [x radius= 3.35, y radius= 3.35]   ;
			\draw [shift={(392.13,98.33)}, rotate = 142.57] [color={rgb, 255:red, 0; green, 0; blue, 0 }  ][fill={rgb, 255:red, 0; green, 0; blue, 0 }  ][line width=0.75]      (0, 0) circle [x radius= 3.35, y radius= 3.35]   ;
			\draw [color={rgb, 255:red, 255; green, 0; blue, 0 }  ,draw opacity=1 ][line width=1.5]    (429.13,104.33) -- (392.13,98.33) ;
			\draw [color={rgb, 255:red, 255; green, 0; blue, 0 }  ,draw opacity=1 ][line width=1.5]    (460.13,96.33) -- (445.13,134.33) ;
			\draw [color={rgb, 255:red, 255; green, 0; blue, 0 }  ,draw opacity=1 ][line width=1.5]    (307.13,111.33) -- (299.13,90.33) ;
			\draw [color={rgb, 255:red, 255; green, 0; blue, 0 }  ,draw opacity=1 ][line width=1.5]    (308.13,58.33) -- (301.13,73.33) ;
			\draw [color={rgb, 255:red, 255; green, 0; blue, 0 }  ,draw opacity=1 ][line width=1.5]    (346.13,87.33) -- (324.13,37.33) ;
			\draw    (392.13,98.33) -- (227.13,180.67) ;
			\draw [shift={(227.13,180.67)}, rotate = 153.48] [color={rgb, 255:red, 0; green, 0; blue, 0 }  ][fill={rgb, 255:red, 0; green, 0; blue, 0 }  ][line width=0.75]      (0, 0) circle [x radius= 3.35, y radius= 3.35]   ;
			\draw [shift={(392.13,98.33)}, rotate = 153.48] [color={rgb, 255:red, 0; green, 0; blue, 0 }  ][fill={rgb, 255:red, 0; green, 0; blue, 0 }  ][line width=0.75]      (0, 0) circle [x radius= 3.35, y radius= 3.35]   ;
			\draw    (299.13,90.33) -- (227.13,180.67) ;
			\draw [shift={(227.13,180.67)}, rotate = 128.56] [color={rgb, 255:red, 0; green, 0; blue, 0 }  ][fill={rgb, 255:red, 0; green, 0; blue, 0 }  ][line width=0.75]      (0, 0) circle [x radius= 3.35, y radius= 3.35]   ;
			\draw [shift={(299.13,90.33)}, rotate = 128.56] [color={rgb, 255:red, 0; green, 0; blue, 0 }  ][fill={rgb, 255:red, 0; green, 0; blue, 0 }  ][line width=0.75]      (0, 0) circle [x radius= 3.35, y radius= 3.35]   ;
			\draw    (299.13,90.33) -- (178.05,155.67) ;
			\draw [shift={(178.05,155.67)}, rotate = 151.65] [color={rgb, 255:red, 0; green, 0; blue, 0 }  ][fill={rgb, 255:red, 0; green, 0; blue, 0 }  ][line width=0.75]      (0, 0) circle [x radius= 3.35, y radius= 3.35]   ;
			\draw [shift={(299.13,90.33)}, rotate = 151.65] [color={rgb, 255:red, 0; green, 0; blue, 0 }  ][fill={rgb, 255:red, 0; green, 0; blue, 0 }  ][line width=0.75]      (0, 0) circle [x radius= 3.35, y radius= 3.35]   ;
			\draw    (299.13,90.33) -- (199.37,111.25) ;
			\draw [shift={(199.37,111.25)}, rotate = 168.16] [color={rgb, 255:red, 0; green, 0; blue, 0 }  ][fill={rgb, 255:red, 0; green, 0; blue, 0 }  ][line width=0.75]      (0, 0) circle [x radius= 3.35, y radius= 3.35]   ;
			\draw [shift={(299.13,90.33)}, rotate = 168.16] [color={rgb, 255:red, 0; green, 0; blue, 0 }  ][fill={rgb, 255:red, 0; green, 0; blue, 0 }  ][line width=0.75]      (0, 0) circle [x radius= 3.35, y radius= 3.35]   ;
			\draw    (308.13,58.33) -- (154.96,89.93) ;
			\draw [shift={(154.96,89.93)}, rotate = 168.35] [color={rgb, 255:red, 0; green, 0; blue, 0 }  ][fill={rgb, 255:red, 0; green, 0; blue, 0 }  ][line width=0.75]      (0, 0) circle [x radius= 3.35, y radius= 3.35]   ;
			\draw [shift={(308.13,58.33)}, rotate = 168.35] [color={rgb, 255:red, 0; green, 0; blue, 0 }  ][fill={rgb, 255:red, 0; green, 0; blue, 0 }  ][line width=0.75]      (0, 0) circle [x radius= 3.35, y radius= 3.35]   ;
			\draw    (110.54,68.61) ;
			\draw [shift={(110.54,68.61)}, rotate = 0] [color={rgb, 255:red, 0; green, 0; blue, 0 }  ][fill={rgb, 255:red, 0; green, 0; blue, 0 }  ][line width=0.75]      (0, 0) circle [x radius= 3.35, y radius= 3.35]   ;
			\draw [shift={(110.54,68.61)}, rotate = 0] [color={rgb, 255:red, 0; green, 0; blue, 0 }  ][fill={rgb, 255:red, 0; green, 0; blue, 0 }  ][line width=0.75]      (0, 0) circle [x radius= 3.35, y radius= 3.35]   ;
			\draw    (324.13,37.33) ;
			\draw [shift={(324.13,37.33)}, rotate = 0] [color={rgb, 255:red, 0; green, 0; blue, 0 }  ][fill={rgb, 255:red, 0; green, 0; blue, 0 }  ][line width=0.75]      (0, 0) circle [x radius= 3.35, y radius= 3.35]   ;
			\draw [shift={(324.13,37.33)}, rotate = 0] [color={rgb, 255:red, 0; green, 0; blue, 0 }  ][fill={rgb, 255:red, 0; green, 0; blue, 0 }  ][line width=0.75]      (0, 0) circle [x radius= 3.35, y radius= 3.35]   ;
			\draw    (301.13,73.33) ;
			\draw [shift={(301.13,73.33)}, rotate = 0] [color={rgb, 255:red, 0; green, 0; blue, 0 }  ][fill={rgb, 255:red, 0; green, 0; blue, 0 }  ][line width=0.75]      (0, 0) circle [x radius= 3.35, y radius= 3.35]   ;
			\draw [shift={(301.13,73.33)}, rotate = 0] [color={rgb, 255:red, 0; green, 0; blue, 0 }  ][fill={rgb, 255:red, 0; green, 0; blue, 0 }  ][line width=0.75]      (0, 0) circle [x radius= 3.35, y radius= 3.35]   ;
			\draw    (163.64,200.35) ;
			\draw [shift={(163.64,200.35)}, rotate = 0] [color={rgb, 255:red, 0; green, 0; blue, 0 }  ][fill={rgb, 255:red, 0; green, 0; blue, 0 }  ][line width=0.75]      (0, 0) circle [x radius= 3.35, y radius= 3.35]   ;
			\draw [shift={(163.64,200.35)}, rotate = 0] [color={rgb, 255:red, 0; green, 0; blue, 0 }  ][fill={rgb, 255:red, 0; green, 0; blue, 0 }  ][line width=0.75]      (0, 0) circle [x radius= 3.35, y radius= 3.35]   ;
			\draw    (133.64,134.35) ;
			\draw [shift={(133.64,134.35)}, rotate = 0] [color={rgb, 255:red, 0; green, 0; blue, 0 }  ][fill={rgb, 255:red, 0; green, 0; blue, 0 }  ][line width=0.75]      (0, 0) circle [x radius= 3.35, y radius= 3.35]   ;
			\draw [shift={(133.64,134.35)}, rotate = 0] [color={rgb, 255:red, 0; green, 0; blue, 0 }  ][fill={rgb, 255:red, 0; green, 0; blue, 0 }  ][line width=0.75]      (0, 0) circle [x radius= 3.35, y radius= 3.35]   ;
			\draw    (89.22,113.03) ;
			\draw [shift={(89.22,113.03)}, rotate = 0] [color={rgb, 255:red, 0; green, 0; blue, 0 }  ][fill={rgb, 255:red, 0; green, 0; blue, 0 }  ][line width=0.75]      (0, 0) circle [x radius= 3.35, y radius= 3.35]   ;
			\draw [shift={(89.22,113.03)}, rotate = 0] [color={rgb, 255:red, 0; green, 0; blue, 0 }  ][fill={rgb, 255:red, 0; green, 0; blue, 0 }  ][line width=0.75]      (0, 0) circle [x radius= 3.35, y radius= 3.35]   ;
			\draw    (307.13,111.33) ;
			\draw [shift={(307.13,111.33)}, rotate = 0] [color={rgb, 255:red, 0; green, 0; blue, 0 }  ][fill={rgb, 255:red, 0; green, 0; blue, 0 }  ][line width=0.75]      (0, 0) circle [x radius= 3.35, y radius= 3.35]   ;
			\draw [shift={(307.13,111.33)}, rotate = 0] [color={rgb, 255:red, 0; green, 0; blue, 0 }  ][fill={rgb, 255:red, 0; green, 0; blue, 0 }  ][line width=0.75]      (0, 0) circle [x radius= 3.35, y radius= 3.35]   ;
			\draw    (346.13,87.33) ;
			\draw [shift={(346.13,87.33)}, rotate = 0] [color={rgb, 255:red, 0; green, 0; blue, 0 }  ][fill={rgb, 255:red, 0; green, 0; blue, 0 }  ][line width=0.75]      (0, 0) circle [x radius= 3.35, y radius= 3.35]   ;
			\draw [shift={(346.13,87.33)}, rotate = 0] [color={rgb, 255:red, 0; green, 0; blue, 0 }  ][fill={rgb, 255:red, 0; green, 0; blue, 0 }  ][line width=0.75]      (0, 0) circle [x radius= 3.35, y radius= 3.35]   ;
			\draw    (460.13,96.33) ;
			\draw [shift={(460.13,96.33)}, rotate = 0] [color={rgb, 255:red, 0; green, 0; blue, 0 }  ][fill={rgb, 255:red, 0; green, 0; blue, 0 }  ][line width=0.75]      (0, 0) circle [x radius= 3.35, y radius= 3.35]   ;
			\draw [shift={(460.13,96.33)}, rotate = 0] [color={rgb, 255:red, 0; green, 0; blue, 0 }  ][fill={rgb, 255:red, 0; green, 0; blue, 0 }  ][line width=0.75]      (0, 0) circle [x radius= 3.35, y radius= 3.35]   ;
			\draw    (429.13,104.33) ;
			\draw [shift={(429.13,104.33)}, rotate = 0] [color={rgb, 255:red, 0; green, 0; blue, 0 }  ][fill={rgb, 255:red, 0; green, 0; blue, 0 }  ][line width=0.75]      (0, 0) circle [x radius= 3.35, y radius= 3.35]   ;
			\draw [shift={(429.13,104.33)}, rotate = 0] [color={rgb, 255:red, 0; green, 0; blue, 0 }  ][fill={rgb, 255:red, 0; green, 0; blue, 0 }  ][line width=0.75]      (0, 0) circle [x radius= 3.35, y radius= 3.35]   ;
			\draw    (445.13,134.33) ;
			\draw [shift={(445.13,134.33)}, rotate = 0] [color={rgb, 255:red, 0; green, 0; blue, 0 }  ][fill={rgb, 255:red, 0; green, 0; blue, 0 }  ][line width=0.75]      (0, 0) circle [x radius= 3.35, y radius= 3.35]   ;
			\draw [shift={(445.13,134.33)}, rotate = 0] [color={rgb, 255:red, 0; green, 0; blue, 0 }  ][fill={rgb, 255:red, 0; green, 0; blue, 0 }  ][line width=0.75]      (0, 0) circle [x radius= 3.35, y radius= 3.35]   ;
			\draw    (398.13,210.67) ;
			\draw [shift={(398.13,210.67)}, rotate = 0] [color={rgb, 255:red, 0; green, 0; blue, 0 }  ][fill={rgb, 255:red, 0; green, 0; blue, 0 }  ][line width=0.75]      (0, 0) circle [x radius= 3.35, y radius= 3.35]   ;
			\draw [shift={(398.13,210.67)}, rotate = 0] [color={rgb, 255:red, 0; green, 0; blue, 0 }  ][fill={rgb, 255:red, 0; green, 0; blue, 0 }  ][line width=0.75]      (0, 0) circle [x radius= 3.35, y radius= 3.35]   ;
			\draw    (343.05,210.67) ;
			\draw [shift={(343.05,210.67)}, rotate = 0] [color={rgb, 255:red, 0; green, 0; blue, 0 }  ][fill={rgb, 255:red, 0; green, 0; blue, 0 }  ][line width=0.75]      (0, 0) circle [x radius= 3.35, y radius= 3.35]   ;
			\draw [shift={(343.05,210.67)}, rotate = 0] [color={rgb, 255:red, 0; green, 0; blue, 0 }  ][fill={rgb, 255:red, 0; green, 0; blue, 0 }  ][line width=0.75]      (0, 0) circle [x radius= 3.35, y radius= 3.35]   ;
			\draw    (369.13,170.33) ;
			\draw [shift={(369.13,170.33)}, rotate = 0] [color={rgb, 255:red, 0; green, 0; blue, 0 }  ][fill={rgb, 255:red, 0; green, 0; blue, 0 }  ][line width=0.75]      (0, 0) circle [x radius= 3.35, y radius= 3.35]   ;
			\draw [shift={(369.13,170.33)}, rotate = 0] [color={rgb, 255:red, 0; green, 0; blue, 0 }  ][fill={rgb, 255:red, 0; green, 0; blue, 0 }  ][line width=0.75]      (0, 0) circle [x radius= 3.35, y radius= 3.35]   ;
			\draw  [fill={rgb, 255:red, 74; green, 144; blue, 226 }  ,fill opacity=0.1 ][dash pattern={on 0.84pt off 2.51pt}] (86.13,70.67) .. controls (96.13,39.67) and (168.13,77.67) .. (186.13,87.67) .. controls (204.13,97.67) and (231.13,132.67) .. (239.13,165.67) .. controls (247.13,198.67) and (185.13,221.67) .. (161.13,212.67) .. controls (137.13,203.67) and (97.08,151.59) .. (87.13,136.67) .. controls (77.19,121.75) and (76.13,101.67) .. (86.13,70.67) -- cycle ;
			\draw  [fill={rgb, 255:red, 245; green, 166; blue, 35 }  ,fill opacity=0.1 ][dash pattern={on 0.84pt off 2.51pt}] (300.13,154.67) .. controls (308.13,142.67) and (346.13,148.67) .. (362.13,154.67) .. controls (378.13,160.67) and (419.13,185.67) .. (415.13,210.67) .. controls (411.13,235.67) and (367.13,224.67) .. (349.13,223.67) .. controls (331.13,222.67) and (297.13,225.67) .. (289.13,213.67) .. controls (281.13,201.67) and (292.13,166.67) .. (300.13,154.67) -- cycle ;
			\draw  [fill={rgb, 255:red, 245; green, 166; blue, 35 }  ,fill opacity=0.1 ][dash pattern={on 0.84pt off 2.51pt}] (314.13,27.67) .. controls (342.13,24.67) and (390.13,87.67) .. (413.13,89.67) .. controls (436.13,91.67) and (468.13,78.67) .. (471.13,95.67) .. controls (474.13,112.67) and (457.13,147.67) .. (445.13,145.67) .. controls (433.13,143.67) and (422.13,113.67) .. (402.13,108.67) .. controls (382.13,103.67) and (316.13,129.67) .. (296.13,122.67) .. controls (276.13,115.67) and (286.13,30.67) .. (314.13,27.67) -- cycle ;
			
			\draw (139,44.4) node [anchor=north west][inner sep=0.75pt]    {$B$};
			\draw (404,60.4) node [anchor=north west][inner sep=0.75pt]    {$G_{1}$};
			\draw (264,192.4) node [anchor=north west][inner sep=0.75pt]    {$G_{2}$};

		\end{tikzpicture}

	\end{center}
	\caption{Example of \cref{x14}: in the Figure, a BAB-graph $G=(B,G_{1},G_{2})$ is shown, where
		$
		\det(G)=\det(B)\det(G_{1})\det(G_{2})=4\cdot8\cdot(-2)=-64.
		$
	}
	\label{Figura2}
\end{figure}
\begin{theorem}
	\label{15} A graph $G$ has a Sachs subgraph if and only if $\textnormal{ker}(G)=\emptyset$.
\end{theorem}

\begin{proof}
	Suppose that $G$ has no Sachs subgraph. Then, by \cref{12}, there exists a set $S\s V(G)$ such that $i(G-S)>\left|S\right|$.
	Let $I$ be the set of isolated vertices of $G-S$. Then $N(I)\s I$, and therefore
	\[
	d(G)\ge\left|I\right|-\left|N(I)\right|\ge i(G-S)-\left|S\right|>0.
	\]
	\noindent If $\ker(G)=\emptyset$, by \cref{9} we have $d(G)=\left|\emptyset\right|-\left|N(\emptyset)\right|=0$, a contradiction.
	
	Conversely, suppose that $G$ has a Sachs subgraph. Let $S\s V(G)$ be an independent set. Then
	\[
	\left|S\right|-\left|N(S)\right|\le i(G-S)-\left|S\right|\le0.
	\]
	\noindent That is, $\emptyset$ is a critical independent set, hence $\textnormal{ker}(G)=\emptyset$.
\end{proof}

As a corollary of \cref{x14} and \cref{15} we have the following.

\begin{corollary}
	Let $G=(B,G_{1},\dots,G_{k})$ be a BAB-graph. If any of the graphs $B,G_{1},\dots,G_{k}$ is singular, then $\textnormal{ker}(G)\neq\emptyset.$
\end{corollary}

It is easy to see that in every BAB-graph $G=(B,G_{1},\dots,G_{k})$, we have $\textnormal{def}(G)\ge\textnormal{def}(B)+k$. The next result shows that equality holds when $G$ has a Sachs subgraph.

\begin{theorem}
	\label{fe2} Let $G=(B,G_{1},\dots,G_{k})$ be a BAB-graph with at least one Sachs subgraph. Then
	\[
	\textnormal{def}(G)=\textnormal{def}(B)+k.
	\]
\end{theorem}

\begin{proof}
	If $G$ has a Sachs subgraph $H$, then by \cref{14}, $H_{1}=(V(H)\ii V(G_{1}),E(H)\ii E(G_{1}))$ is a Sachs subgraph of $G_{1}$. It suffices to show that $\textnormal{def}(G_{1})=1$.
	Let $M_{1}$ be a matching of $H_{1}$ obtained by choosing a perfect matching in each even component of $H_{1}$ and a near-perfect matching in each odd component. Note that since $G_{1}$ is an almost bipartite graph, $H_{1}$ has at most one odd component. On the other hand, if $H_{1}$ has no odd components, then $M_{1}$ is a perfect matching of $G_{1}$, a contradiction. Therefore $M_{1}$ leaves exactly one vertex unmatched in $H_{1}$, and hence $\textnormal{def}(G_{1})=1$. Repeat this argument for the other graphs $G_{2},\dots,G_{k}$ to complete the proof.
\end{proof}

Equivalently, from \cref{fe2} we can write
\[
\mu(G)=\frac{\left|G\right|-\left|B\right|-k}{2}+\mu(B).
\]

\begin{corollary}
	Let $G=(B,G_{1},\dots,G_{k})$ be a nonsingular BAB-graph. Then
	\[
	\textnormal{def}(G)=\textnormal{def}(B)+k.
	\]
\end{corollary}

The converse of \cref{fe2} is not true, as $B(G)$ may have no Sachs subgraphs while the graphs $G_{1},\dots,G_{k}$ do. However, when some graph $G_{i}$ has no Sachs subgraph, the same reasoning in the proof of \cref{fe2} shows that $\textnormal{def}(G_{i})>1$.

\section{Independent Set Structure in BAB-Graphs}\label{sss4}

In an almost bipartite non-König--Egerváry graph or in an $R$-disjoint graph, we have $\textnormal{ker}(G)=\textnormal{core}(G)$. In an $R$-disjoint graph, it holds that $\left|\textnormal{corona}(G)\right|+\left|\textnormal{ker}(G)\right|=2\alpha(G)+k$. In a BAB-graph, these results do not hold; see, for example, \cref{Figura3}.

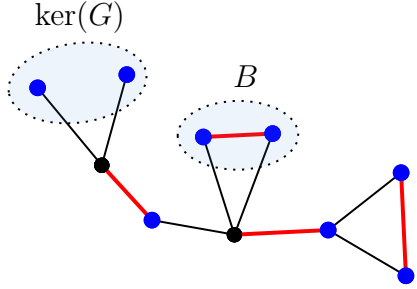
\begin{figure}[h!]
	
	\begin{center}

		\tikzset{every picture/.style={line width=0.75pt}} 
		
		\begin{tikzpicture}[x=0.75pt,y=0.75pt,yscale=-1,xscale=1]
			
			\draw    (198.2,137.55) -- (165.85,98.47) ;
			\draw [shift={(165.85,98.47)}, rotate = 230.38] [color={rgb, 255:red, 0; green, 0; blue, 0 }  ][fill={rgb, 255:red, 0; green, 0; blue, 0 }  ][line width=0.75]      (0, 0) circle [x radius= 3.35, y radius= 3.35]   ;
			\draw [shift={(198.2,137.55)}, rotate = 230.38] [color={rgb, 255:red, 0; green, 0; blue, 0 }  ][fill={rgb, 255:red, 0; green, 0; blue, 0 }  ][line width=0.75]      (0, 0) circle [x radius= 3.35, y radius= 3.35]   ;
			\draw    (210.83,91.8) -- (198.2,137.55) ;
			\draw [shift={(198.2,137.55)}, rotate = 105.43] [color={rgb, 255:red, 0; green, 0; blue, 0 }  ][fill={rgb, 255:red, 0; green, 0; blue, 0 }  ][line width=0.75]      (0, 0) circle [x radius= 3.35, y radius= 3.35]   ;
			\draw [shift={(210.83,91.8)}, rotate = 105.43] [color={rgb, 255:red, 0; green, 0; blue, 0 }  ][fill={rgb, 255:red, 0; green, 0; blue, 0 }  ][line width=0.75]      (0, 0) circle [x radius= 3.35, y radius= 3.35]   ;
			\draw    (223.53,165.25) -- (198.2,137.55) ;
			\draw [shift={(198.2,137.55)}, rotate = 227.55] [color={rgb, 255:red, 0; green, 0; blue, 0 }  ][fill={rgb, 255:red, 0; green, 0; blue, 0 }  ][line width=0.75]      (0, 0) circle [x radius= 3.35, y radius= 3.35]   ;
			\draw [shift={(223.53,165.25)}, rotate = 227.55] [color={rgb, 255:red, 0; green, 0; blue, 0 }  ][fill={rgb, 255:red, 0; green, 0; blue, 0 }  ][line width=0.75]      (0, 0) circle [x radius= 3.35, y radius= 3.35]   ;
			\draw    (265.1,172.58) -- (223.53,165.25) ;
			\draw [shift={(223.53,165.25)}, rotate = 190.01] [color={rgb, 255:red, 0; green, 0; blue, 0 }  ][fill={rgb, 255:red, 0; green, 0; blue, 0 }  ][line width=0.75]      (0, 0) circle [x radius= 3.35, y radius= 3.35]   ;
			\draw [shift={(265.1,172.58)}, rotate = 190.01] [color={rgb, 255:red, 0; green, 0; blue, 0 }  ][fill={rgb, 255:red, 0; green, 0; blue, 0 }  ][line width=0.75]      (0, 0) circle [x radius= 3.35, y radius= 3.35]   ;
			\draw    (312.39,170.25) -- (265.1,172.58) ;
			\draw [shift={(265.1,172.58)}, rotate = 177.17] [color={rgb, 255:red, 0; green, 0; blue, 0 }  ][fill={rgb, 255:red, 0; green, 0; blue, 0 }  ][line width=0.75]      (0, 0) circle [x radius= 3.35, y radius= 3.35]   ;
			\draw [shift={(312.39,170.25)}, rotate = 177.17] [color={rgb, 255:red, 0; green, 0; blue, 0 }  ][fill={rgb, 255:red, 0; green, 0; blue, 0 }  ][line width=0.75]      (0, 0) circle [x radius= 3.35, y radius= 3.35]   ;
			\draw    (349.17,141.31) -- (312.39,170.25) ;
			\draw [shift={(312.39,170.25)}, rotate = 141.8] [color={rgb, 255:red, 0; green, 0; blue, 0 }  ][fill={rgb, 255:red, 0; green, 0; blue, 0 }  ][line width=0.75]      (0, 0) circle [x radius= 3.35, y radius= 3.35]   ;
			\draw [shift={(349.17,141.31)}, rotate = 141.8] [color={rgb, 255:red, 0; green, 0; blue, 0 }  ][fill={rgb, 255:red, 0; green, 0; blue, 0 }  ][line width=0.75]      (0, 0) circle [x radius= 3.35, y radius= 3.35]   ;
			\draw    (351.6,193.35) -- (312.39,170.25) ;
			\draw [shift={(312.39,170.25)}, rotate = 210.51] [color={rgb, 255:red, 0; green, 0; blue, 0 }  ][fill={rgb, 255:red, 0; green, 0; blue, 0 }  ][line width=0.75]      (0, 0) circle [x radius= 3.35, y radius= 3.35]   ;
			\draw [shift={(351.6,193.35)}, rotate = 210.51] [color={rgb, 255:red, 0; green, 0; blue, 0 }  ][fill={rgb, 255:red, 0; green, 0; blue, 0 }  ][line width=0.75]      (0, 0) circle [x radius= 3.35, y radius= 3.35]   ;
			\draw    (349.17,141.31) -- (351.6,193.35) ;
			\draw [shift={(351.6,193.35)}, rotate = 87.33] [color={rgb, 255:red, 0; green, 0; blue, 0 }  ][fill={rgb, 255:red, 0; green, 0; blue, 0 }  ][line width=0.75]      (0, 0) circle [x radius= 3.35, y radius= 3.35]   ;
			\draw [shift={(349.17,141.31)}, rotate = 87.33] [color={rgb, 255:red, 0; green, 0; blue, 0 }  ][fill={rgb, 255:red, 0; green, 0; blue, 0 }  ][line width=0.75]      (0, 0) circle [x radius= 3.35, y radius= 3.35]   ;
			\draw    (249.4,123.28) -- (265.1,172.58) ;
			\draw [shift={(265.1,172.58)}, rotate = 72.34] [color={rgb, 255:red, 0; green, 0; blue, 0 }  ][fill={rgb, 255:red, 0; green, 0; blue, 0 }  ][line width=0.75]      (0, 0) circle [x radius= 3.35, y radius= 3.35]   ;
			\draw [shift={(249.4,123.28)}, rotate = 72.34] [color={rgb, 255:red, 0; green, 0; blue, 0 }  ][fill={rgb, 255:red, 0; green, 0; blue, 0 }  ][line width=0.75]      (0, 0) circle [x radius= 3.35, y radius= 3.35]   ;
			\draw    (284.36,121.56) -- (265.1,172.58) ;
			\draw [shift={(265.1,172.58)}, rotate = 110.68] [color={rgb, 255:red, 0; green, 0; blue, 0 }  ][fill={rgb, 255:red, 0; green, 0; blue, 0 }  ][line width=0.75]      (0, 0) circle [x radius= 3.35, y radius= 3.35]   ;
			\draw [shift={(284.36,121.56)}, rotate = 110.68] [color={rgb, 255:red, 0; green, 0; blue, 0 }  ][fill={rgb, 255:red, 0; green, 0; blue, 0 }  ][line width=0.75]      (0, 0) circle [x radius= 3.35, y radius= 3.35]   ;
			\draw    (284.36,121.56) -- (249.4,123.28) ;
			\draw [shift={(249.4,123.28)}, rotate = 177.17] [color={rgb, 255:red, 0; green, 0; blue, 0 }  ][fill={rgb, 255:red, 0; green, 0; blue, 0 }  ][line width=0.75]      (0, 0) circle [x radius= 3.35, y radius= 3.35]   ;
			\draw [shift={(284.36,121.56)}, rotate = 177.17] [color={rgb, 255:red, 0; green, 0; blue, 0 }  ][fill={rgb, 255:red, 0; green, 0; blue, 0 }  ][line width=0.75]      (0, 0) circle [x radius= 3.35, y radius= 3.35]   ;
			\draw [color={rgb, 255:red, 255; green, 0; blue, 0 }  ,draw opacity=1 ][line width=1.5]    (349.17,141.31) -- (351.6,193.35) ;
			\draw [color={rgb, 255:red, 255; green, 0; blue, 0 }  ,draw opacity=1 ][line width=1.5]    (312.39,170.25) -- (265.1,172.58) ;
			\draw [color={rgb, 255:red, 255; green, 0; blue, 0 }  ,draw opacity=1 ][line width=1.5]    (284.36,121.56) -- (249.4,123.28) ;
			\draw [color={rgb, 255:red, 255; green, 0; blue, 0 }  ,draw opacity=1 ][line width=1.5]    (223.53,165.25) -- (198.2,137.55) ;
			\draw    (198.2,137.55) ;
			\draw [shift={(198.2,137.55)}, rotate = 0] [color={rgb, 255:red, 0; green, 0; blue, 0 }  ][fill={rgb, 255:red, 0; green, 0; blue, 0 }  ][line width=0.75]      (0, 0) circle [x radius= 3.35, y radius= 3.35]   ;
			\draw [shift={(198.2,137.55)}, rotate = 0] [color={rgb, 255:red, 0; green, 0; blue, 0 }  ][fill={rgb, 255:red, 0; green, 0; blue, 0 }  ][line width=0.75]      (0, 0) circle [x radius= 3.35, y radius= 3.35]   ;
			\draw    (223.53,165.25) ;
			\draw [shift={(223.53,165.25)}, rotate = 0] [color={rgb, 255:red, 0; green, 0; blue, 0 }  ][fill={rgb, 255:red, 0; green, 0; blue, 0 }  ][line width=0.75]      (0, 0) circle [x radius= 3.35, y radius= 3.35]   ;
			\draw [shift={(223.53,165.25)}, rotate = 0] [color={rgb, 255:red, 0; green, 0; blue, 0 }  ][fill={rgb, 255:red, 0; green, 0; blue, 0 }  ][line width=0.75]      (0, 0) circle [x radius= 3.35, y radius= 3.35]   ;
			\draw    (349.17,141.31) ;
			\draw [shift={(349.17,141.31)}, rotate = 0] [color={rgb, 255:red, 0; green, 0; blue, 0 }  ][fill={rgb, 255:red, 0; green, 0; blue, 0 }  ][line width=0.75]      (0, 0) circle [x radius= 3.35, y radius= 3.35]   ;
			\draw [shift={(349.17,141.31)}, rotate = 0] [color={rgb, 255:red, 0; green, 0; blue, 0 }  ][fill={rgb, 255:red, 0; green, 0; blue, 0 }  ][line width=0.75]      (0, 0) circle [x radius= 3.35, y radius= 3.35]   ;
			\draw    (312.39,170.25) ;
			\draw [shift={(312.39,170.25)}, rotate = 0] [color={rgb, 255:red, 0; green, 0; blue, 0 }  ][fill={rgb, 255:red, 0; green, 0; blue, 0 }  ][line width=0.75]      (0, 0) circle [x radius= 3.35, y radius= 3.35]   ;
			\draw [shift={(312.39,170.25)}, rotate = 0] [color={rgb, 255:red, 0; green, 0; blue, 0 }  ][fill={rgb, 255:red, 0; green, 0; blue, 0 }  ][line width=0.75]      (0, 0) circle [x radius= 3.35, y radius= 3.35]   ;
			\draw    (284.36,121.56) ;
			\draw [shift={(284.36,121.56)}, rotate = 0] [color={rgb, 255:red, 0; green, 0; blue, 0 }  ][fill={rgb, 255:red, 0; green, 0; blue, 0 }  ][line width=0.75]      (0, 0) circle [x radius= 3.35, y radius= 3.35]   ;
			\draw [shift={(284.36,121.56)}, rotate = 0] [color={rgb, 255:red, 0; green, 0; blue, 0 }  ][fill={rgb, 255:red, 0; green, 0; blue, 0 }  ][line width=0.75]      (0, 0) circle [x radius= 3.35, y radius= 3.35]   ;
			\draw    (265.1,172.58) ;
			\draw [shift={(265.1,172.58)}, rotate = 0] [color={rgb, 255:red, 0; green, 0; blue, 0 }  ][fill={rgb, 255:red, 0; green, 0; blue, 0 }  ][line width=0.75]      (0, 0) circle [x radius= 3.35, y radius= 3.35]   ;
			\draw [shift={(265.1,172.58)}, rotate = 0] [color={rgb, 255:red, 0; green, 0; blue, 0 }  ][fill={rgb, 255:red, 0; green, 0; blue, 0 }  ][line width=0.75]      (0, 0) circle [x radius= 3.35, y radius= 3.35]   ;
			\draw    (249.4,123.28) ;
			\draw [shift={(249.4,123.28)}, rotate = 0] [color={rgb, 255:red, 0; green, 0; blue, 0 }  ][fill={rgb, 255:red, 0; green, 0; blue, 0 }  ][line width=0.75]      (0, 0) circle [x radius= 3.35, y radius= 3.35]   ;
			\draw [shift={(249.4,123.28)}, rotate = 0] [color={rgb, 255:red, 0; green, 0; blue, 0 }  ][fill={rgb, 255:red, 0; green, 0; blue, 0 }  ][line width=0.75]      (0, 0) circle [x radius= 3.35, y radius= 3.35]   ;
			\draw    (351.6,193.35) ;
			\draw [shift={(351.6,193.35)}, rotate = 0] [color={rgb, 255:red, 0; green, 0; blue, 0 }  ][fill={rgb, 255:red, 0; green, 0; blue, 0 }  ][line width=0.75]      (0, 0) circle [x radius= 3.35, y radius= 3.35]   ;
			\draw [shift={(351.6,193.35)}, rotate = 0] [color={rgb, 255:red, 0; green, 0; blue, 0 }  ][fill={rgb, 255:red, 0; green, 0; blue, 0 }  ][line width=0.75]      (0, 0) circle [x radius= 3.35, y radius= 3.35]   ;
			\draw  [fill={rgb, 255:red, 74; green, 144; blue, 226 }  ,fill opacity=0.1 ][dash pattern={on 0.84pt off 2.51pt}] (151.33,100.79) .. controls (149.82,89.85) and (164.11,78.84) .. (183.26,76.19) .. controls (202.41,73.54) and (219.16,80.27) .. (220.67,91.21) .. controls (222.18,102.15) and (207.89,113.16) .. (188.74,115.81) .. controls (169.59,118.46) and (152.84,111.73) .. (151.33,100.79) -- cycle ;
			\draw [color={rgb, 255:red, 0; green, 0; blue, 255 }  ,draw opacity=1 ]   (249.4,123.28) ;
			\draw [shift={(249.4,123.28)}, rotate = 0] [color={rgb, 255:red, 0; green, 0; blue, 255 }  ,draw opacity=1 ][fill={rgb, 255:red, 0; green, 0; blue, 255 }  ,fill opacity=1 ][line width=0.75]      (0, 0) circle [x radius= 3.69, y radius= 3.69]   ;
			\draw [shift={(249.4,123.28)}, rotate = 0] [color={rgb, 255:red, 0; green, 0; blue, 255 }  ,draw opacity=1 ][fill={rgb, 255:red, 0; green, 0; blue, 255 }  ,fill opacity=1 ][line width=0.75]      (0, 0) circle [x radius= 3.69, y radius= 3.69]   ;
			\draw [color={rgb, 255:red, 0; green, 0; blue, 255 }  ,draw opacity=1 ]   (165.85,98.47) ;
			\draw [shift={(165.85,98.47)}, rotate = 0] [color={rgb, 255:red, 0; green, 0; blue, 255 }  ,draw opacity=1 ][fill={rgb, 255:red, 0; green, 0; blue, 255 }  ,fill opacity=1 ][line width=0.75]      (0, 0) circle [x radius= 3.69, y radius= 3.69]   ;
			\draw [shift={(165.85,98.47)}, rotate = 0] [color={rgb, 255:red, 0; green, 0; blue, 255 }  ,draw opacity=1 ][fill={rgb, 255:red, 0; green, 0; blue, 255 }  ,fill opacity=1 ][line width=0.75]      (0, 0) circle [x radius= 3.69, y radius= 3.69]   ;
			\draw [color={rgb, 255:red, 0; green, 0; blue, 255 }  ,draw opacity=1 ]   (210.83,91.8) ;
			\draw [shift={(210.83,91.8)}, rotate = 0] [color={rgb, 255:red, 0; green, 0; blue, 255 }  ,draw opacity=1 ][fill={rgb, 255:red, 0; green, 0; blue, 255 }  ,fill opacity=1 ][line width=0.75]      (0, 0) circle [x radius= 3.69, y radius= 3.69]   ;
			\draw [shift={(210.83,91.8)}, rotate = 0] [color={rgb, 255:red, 0; green, 0; blue, 255 }  ,draw opacity=1 ][fill={rgb, 255:red, 0; green, 0; blue, 255 }  ,fill opacity=1 ][line width=0.75]      (0, 0) circle [x radius= 3.69, y radius= 3.69]   ;
			\draw [color={rgb, 255:red, 0; green, 0; blue, 255 }  ,draw opacity=1 ]   (223.53,165.25) ;
			\draw [shift={(223.53,165.25)}, rotate = 0] [color={rgb, 255:red, 0; green, 0; blue, 255 }  ,draw opacity=1 ][fill={rgb, 255:red, 0; green, 0; blue, 255 }  ,fill opacity=1 ][line width=0.75]      (0, 0) circle [x radius= 3.69, y radius= 3.69]   ;
			\draw [shift={(223.53,165.25)}, rotate = 0] [color={rgb, 255:red, 0; green, 0; blue, 255 }  ,draw opacity=1 ][fill={rgb, 255:red, 0; green, 0; blue, 255 }  ,fill opacity=1 ][line width=0.75]      (0, 0) circle [x radius= 3.69, y radius= 3.69]   ;
			\draw [color={rgb, 255:red, 0; green, 0; blue, 255 }  ,draw opacity=1 ]   (351.6,193.35) ;
			\draw [shift={(351.6,193.35)}, rotate = 0] [color={rgb, 255:red, 0; green, 0; blue, 255 }  ,draw opacity=1 ][fill={rgb, 255:red, 0; green, 0; blue, 255 }  ,fill opacity=1 ][line width=0.75]      (0, 0) circle [x radius= 3.69, y radius= 3.69]   ;
			\draw [shift={(351.6,193.35)}, rotate = 0] [color={rgb, 255:red, 0; green, 0; blue, 255 }  ,draw opacity=1 ][fill={rgb, 255:red, 0; green, 0; blue, 255 }  ,fill opacity=1 ][line width=0.75]      (0, 0) circle [x radius= 3.69, y radius= 3.69]   ;
			\draw [color={rgb, 255:red, 0; green, 0; blue, 255 }  ,draw opacity=1 ]   (312.39,170.25) ;
			\draw [shift={(312.39,170.25)}, rotate = 0] [color={rgb, 255:red, 0; green, 0; blue, 255 }  ,draw opacity=1 ][fill={rgb, 255:red, 0; green, 0; blue, 255 }  ,fill opacity=1 ][line width=0.75]      (0, 0) circle [x radius= 3.69, y radius= 3.69]   ;
			\draw [shift={(312.39,170.25)}, rotate = 0] [color={rgb, 255:red, 0; green, 0; blue, 255 }  ,draw opacity=1 ][fill={rgb, 255:red, 0; green, 0; blue, 255 }  ,fill opacity=1 ][line width=0.75]      (0, 0) circle [x radius= 3.69, y radius= 3.69]   ;
			\draw [color={rgb, 255:red, 0; green, 0; blue, 255 }  ,draw opacity=1 ]   (349.17,141.31) ;
			\draw [shift={(349.17,141.31)}, rotate = 0] [color={rgb, 255:red, 0; green, 0; blue, 255 }  ,draw opacity=1 ][fill={rgb, 255:red, 0; green, 0; blue, 255 }  ,fill opacity=1 ][line width=0.75]      (0, 0) circle [x radius= 3.69, y radius= 3.69]   ;
			\draw [shift={(349.17,141.31)}, rotate = 0] [color={rgb, 255:red, 0; green, 0; blue, 255 }  ,draw opacity=1 ][fill={rgb, 255:red, 0; green, 0; blue, 255 }  ,fill opacity=1 ][line width=0.75]      (0, 0) circle [x radius= 3.69, y radius= 3.69]   ;
			\draw [color={rgb, 255:red, 0; green, 0; blue, 255 }  ,draw opacity=1 ]   (284.36,121.56) ;
			\draw [shift={(284.36,121.56)}, rotate = 0] [color={rgb, 255:red, 0; green, 0; blue, 255 }  ,draw opacity=1 ][fill={rgb, 255:red, 0; green, 0; blue, 255 }  ,fill opacity=1 ][line width=0.75]      (0, 0) circle [x radius= 3.69, y radius= 3.69]   ;
			\draw [shift={(284.36,121.56)}, rotate = 0] [color={rgb, 255:red, 0; green, 0; blue, 255 }  ,draw opacity=1 ][fill={rgb, 255:red, 0; green, 0; blue, 255 }  ,fill opacity=1 ][line width=0.75]      (0, 0) circle [x radius= 3.69, y radius= 3.69]   ;
			\draw  [fill={rgb, 255:red, 74; green, 144; blue, 226 }  ,fill opacity=0.1 ][dash pattern={on 0.84pt off 2.51pt}] (236.49,122.78) .. controls (236.38,113.19) and (249.89,105.26) .. (266.67,105.06) .. controls (283.45,104.86) and (297.15,112.47) .. (297.26,122.06) .. controls (297.37,131.65) and (283.86,139.58) .. (267.08,139.78) .. controls (250.3,139.98) and (236.61,132.37) .. (236.49,122.78) -- cycle ;
			
			\draw (163,53.4) node [anchor=north west][inner sep=0.75pt]    {$\text{ker}( G)$};
			\draw (263,85.4) node [anchor=north west][inner sep=0.75pt]    {$B$};

		\end{tikzpicture}

	\end{center}
	\caption{In this example, a BAB-graph $G=(G,G_{1})$ is shown; in red we display a maximum matching of $G$, and in blue the vertices of $\text{corona}(G)$. Notice that $\left|\text{corona}(G)\right|+\left|\text{ker}(G)\right|=8+2=10<11=2\cdot5+1=2\alpha(G)+k$.
	}
	\label{Figura3}
	
\end{figure}

We will show in \cref{22} that $\left|\textnormal{corona}(G)\right|+\left|\textnormal{ker}(G)\right|\le 2\alpha(G)+k$ for every BAB-graph.

\begin{lemma}
	\label{16} Let $G$ be a graph and $I$ a critical independent set. Then
	\begin{itemize}
		\item If there exists $S\s N(I)$ such that $\left|N(S)\ii I\right|=\left|S\right|$, then $I-N(S)$ is a critical independent set.
		\item Let $S\s N(I)$ be a set of maximum cardinality such that $\left|N(S)\ii I\right|=\left|S\right|$, then $I-N(S)=\textnormal{ker}(G)$.
	\end{itemize}
\end{lemma}

\begin{proof}
	For the first part, we repeat the argument in the proof of \cref{6}.
	Now, suppose that $S\s N(I)$ is a set of maximum cardinality such that $\left|N(S)\ii I\right|=\left|S\right|$. By the first part, $I-N(S)$ is a critical independent set of $G$. Then we define the sets $T_{1}=I-\left(\left(N(S)\ii S\right)\u\textnormal{ker}(G)\right)$
	and $T_{2}=N(I)-\left(S\u N\left(\textnormal{ker}(G)\right)\right)$, see \cref{Figura4}.

	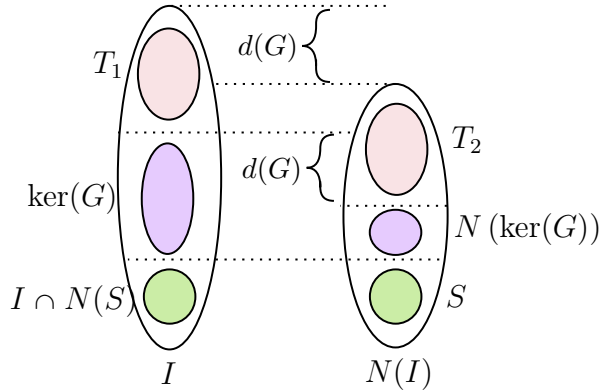
\begin{figure}[h!]
		
		\begin{center}

			\tikzset{every picture/.style={line width=0.75pt}} 
			
			\begin{tikzpicture}[x=0.75pt,y=0.75pt,yscale=-1,xscale=1]
				
				\draw   (152,134.7) .. controls (152,86.82) and (163.67,48) .. (178.07,48) .. controls (192.46,48) and (204.13,86.82) .. (204.13,134.7) .. controls (204.13,182.58) and (192.46,221.4) .. (178.07,221.4) .. controls (163.67,221.4) and (152,182.58) .. (152,134.7) -- cycle ;
				\draw   (265.8,153.9) .. controls (265.8,117.17) and (277.56,87.4) .. (292.07,87.4) .. controls (306.57,87.4) and (318.33,117.17) .. (318.33,153.9) .. controls (318.33,190.63) and (306.57,220.4) .. (292.07,220.4) .. controls (277.56,220.4) and (265.8,190.63) .. (265.8,153.9) -- cycle ;
				\draw  [dash pattern={on 0.84pt off 2.51pt}]  (292.07,87.4) -- (201.13,87.4) ;
				\draw  [dash pattern={on 0.84pt off 2.51pt}]  (287.13,48) -- (178.07,48) ;
				\draw  [dash pattern={on 0.84pt off 2.51pt}]  (314.13,176) -- (157.07,176) ;
				\draw  [dash pattern={on 0.84pt off 2.51pt}]  (318,148.9) -- (264.13,148.9) ;
				\draw  [dash pattern={on 0.84pt off 2.51pt}]  (269.13,111.9) -- (152.13,111.9) ;
				\draw   (261,113) .. controls (256.41,113) and (254.12,115.29) .. (254.12,119.88) -- (254.12,119.88) .. controls (254.12,126.43) and (251.83,129.7) .. (247.25,129.7) .. controls (251.83,129.7) and (254.12,132.97) .. (254.12,139.52)(254.12,136.58) -- (254.12,139.52) .. controls (254.12,144.11) and (256.41,146.4) .. (261,146.4) ;
				\draw   (259,50) .. controls (254.33,50) and (252,52.33) .. (252,57) -- (252,58.2) .. controls (252,64.87) and (249.67,68.2) .. (245,68.2) .. controls (249.67,68.2) and (252,71.53) .. (252,78.2)(252,75.2) -- (252,79.4) .. controls (252,84.07) and (254.33,86.4) .. (259,86.4) ;
				\draw  [fill={rgb, 255:red, 126; green, 211; blue, 33 }  ,fill opacity=0.4 ] (165.13,194.7) .. controls (165.13,187.13) and (170.95,181) .. (178.13,181) .. controls (185.31,181) and (191.13,187.13) .. (191.13,194.7) .. controls (191.13,202.27) and (185.31,208.4) .. (178.13,208.4) .. controls (170.95,208.4) and (165.13,202.27) .. (165.13,194.7) -- cycle ;
				\draw  [fill={rgb, 255:red, 126; green, 211; blue, 33 }  ,fill opacity=0.4 ] (279.13,194.7) .. controls (279.13,187.13) and (284.95,181) .. (292.13,181) .. controls (299.31,181) and (305.13,187.13) .. (305.13,194.7) .. controls (305.13,202.27) and (299.31,208.4) .. (292.13,208.4) .. controls (284.95,208.4) and (279.13,202.27) .. (279.13,194.7) -- cycle ;
				\draw  [fill={rgb, 255:red, 144; green, 19; blue, 254 }  ,fill opacity=0.22 ] (164.07,145.3) .. controls (164.07,129.89) and (169.89,117.4) .. (177.07,117.4) .. controls (184.25,117.4) and (190.07,129.89) .. (190.07,145.3) .. controls (190.07,160.71) and (184.25,173.2) .. (177.07,173.2) .. controls (169.89,173.2) and (164.07,160.71) .. (164.07,145.3) -- cycle ;
				\draw  [fill={rgb, 255:red, 144; green, 19; blue, 254 }  ,fill opacity=0.22 ] (279.07,162.3) .. controls (279.07,156) and (284.89,150.9) .. (292.07,150.9) .. controls (299.25,150.9) and (305.07,156) .. (305.07,162.3) .. controls (305.07,168.6) and (299.25,173.7) .. (292.07,173.7) .. controls (284.89,173.7) and (279.07,168.6) .. (279.07,162.3) -- cycle ;
				\draw  [fill={rgb, 255:red, 208; green, 2; blue, 27 }  ,fill opacity=0.11 ] (162.13,82.4) .. controls (162.13,69.7) and (169.07,59.4) .. (177.63,59.4) .. controls (186.19,59.4) and (193.13,69.7) .. (193.13,82.4) .. controls (193.13,95.1) and (186.19,105.4) .. (177.63,105.4) .. controls (169.07,105.4) and (162.13,95.1) .. (162.13,82.4) -- cycle ;
				\draw  [fill={rgb, 255:red, 208; green, 2; blue, 27 }  ,fill opacity=0.11 ] (277.13,120.4) .. controls (277.13,107.7) and (284.07,97.4) .. (292.63,97.4) .. controls (301.19,97.4) and (308.13,107.7) .. (308.13,120.4) .. controls (308.13,133.1) and (301.19,143.4) .. (292.63,143.4) .. controls (284.07,143.4) and (277.13,133.1) .. (277.13,120.4) -- cycle ;
				
				\draw (213,122.4) node [anchor=north west][inner sep=0.75pt]  [font=\small]  {$d( G)$};
				\draw (211,60.4) node [anchor=north west][inner sep=0.75pt]  [font=\small]  {$d( G)$};
				\draw (138,69.4) node [anchor=north west][inner sep=0.75pt]    {$T_{1}$};
				\draw (105,135.4) node [anchor=north west][inner sep=0.75pt]    {$\text{ker}( G)$};
				\draw (96,187.4) node [anchor=north west][inner sep=0.75pt]    {$I\cap N( S)$};
				\draw (275,224.4) node [anchor=north west][inner sep=0.75pt]    {$N( I)$};
				\draw (172,226.4) node [anchor=north west][inner sep=0.75pt]    {$I$};
				\draw (320,151.3) node [anchor=north west][inner sep=0.75pt]    {$N\left(\text{ker}( G)\right)$};
				\draw (319,107.4) node [anchor=north west][inner sep=0.75pt]    {$T_{2}$};
				\draw (316,187.4) node [anchor=north west][inner sep=0.75pt]    {$S$};

			\end{tikzpicture}

		\end{center}
		\caption{Illustration of the proof of \cref{16}.
		}
		\label{Figura4}
		
	\end{figure}
	\noindent Suppose that $I-N(S)\neq\textnormal{ker}(G)$, then $T_{1}\neq\emptyset$ and
	\[
	\left|I-N(S)\right|-\left|N(I-N(S))\right|=d(G)=\left|\textnormal{ker}(G)\right|-\left|N(\textnormal{ker}(G))\right|,
	\]
	\noindent that is,
	\[
	\left|T_{1}\right|=\left|I-N(S)\right|-\left|\textnormal{ker}(G)\right|=\left|N(I-N(S))\right|-\left|N(\textnormal{ker}(G))\right|=\left|T_{2}\right|.
	\]
	\noindent Hence $T_{2}\neq\emptyset$. Now we define $S^{\prime}=S\u T_{2}$ and obtain
	\begin{eqnarray*}
		\left|N(S^{\prime})\ii I\right| & = & \left|N(S)\ii I\right|+\left|N(T_{2})\ii(I-N(S))\right|\\
		& = & \left|S\right|+\left|T_{1}\right|\\
		& = & \left|S^{\prime}\right|.
	\end{eqnarray*}
	\noindent Which contradicts the maximality of $S$.
	\end{proof}

As a consequence, we have the following Hall-type condition to
characterize when a critical independent set is the kernel of the
graph.

\begin{corollary}
	\label{17} Let $G$ be a graph and $I$ a critical
	independent set. Then $I=\textnormal{ker}(G)$ if and only if
	\[
	\left|S\right|<\left|N(S)\cap I\right|,
	\]
	\noindent for all $S \subseteq N(I)$. 
\end{corollary}

As a second consequence of \cref{16}, \cref{7}, and \cref{10}, we have
an explicit way to find $\textnormal{ker}(G)$ from $A(G)$.

\begin{theorem}
	\label{18} Let $G$ be a BAB-graph and $S \subseteq A(G)$
	of maximum cardinality such that $\left|N(S)\cap\textnormal{nucleus}(G)\right|=\left|S\right|$.
	Then
	\[
	\textnormal{ker}(G)=\textnormal{nucleus}(G)-N(S).
	\]
\end{theorem}

Now we study how maximum independent sets behave
in BAB-graphs, aiming to bound $\left|\text{corona}(G)\right|+\left|\text{ker}(G)\right|$.  

\begin{theorem}
	\label{19} \cite{berge2005some} An independent set
	$S$ is maximum if and only if every independent set disjoint
	from $S$ can be matched into $S$. 
\end{theorem}

\begin{theorem}
	\label{20} Let $G=(B,G_{1},\dots,G_{k})$ be a BAB-graph,
	$U,W$ a bipartition of $G[C(G)]$, and $S$ a maximum independent
	set of $C(G_{1})\cup\dots\cup C(G_{k})$. Then the
	following items hold.
	\begin{itemize}
		\item $\textnormal{nucleus}(G)\cup S \cup U$ is a maximum independent set.
		\item $\textnormal{corona}(G)\cup A(G)=V(G)$.
		\item $\textnormal{core}(G)\subseteq\textnormal{nucleus}(G)$.
		\item $2\alpha(G)+k=\left|D\right|+\left|\textnormal{nucleus}(G)\right|+\left|C\right|$.
	\end{itemize}
\end{theorem}

\begin{proof}
	Reasoning as in the proof of \cref{7}, we can choose
	a maximum matching $M$ of $G$ such that it saturates exactly one vertex
	of each cycle $C(G_{i})$, for $i=1,\dots,k$. Then, by \cref{7}
	and \cref{ge}, $M$ matches the vertices of $A$ into $\textnormal{nucleus}(G)$,
	matches vertices of $C(G_{i})$ into vertices of $C(G_{i})$, and $U$
	into $W$. Hence, it is possible to match the vertices of $V(G)-\left(\textnormal{nucleus}(G)\cup S\cup U\right)=A(G)\cup W\cup V(C(G_{1}))\cup\dots\cup V(G(G_{k}))$
	into $\textnormal{nucleus}(G)\cup S\cup U$. Therefore, any independent set
	in $V(G)-\left(\textnormal{nucleus}(G)\cup S\cup U\right)$ can also
	be matched into $\textnormal{nucleus}(G)\cup S\cup U$. Then, by \cref{19},
	$\textnormal{nucleus}(G)\cup S\cup U$ is a maximum independent set of
	$G$. 
	
	We can modify $\textnormal{nucleus}(G)\cup S\cup U$ by moving the vertices
	in each odd cycle $C(G_{i})$, thus obtaining $V(C(G_{i}))\subseteq\textnormal{corona}(G)$.
	Furthermore, by symmetry, $\textnormal{nucleus}(G)\cup S\cup W$ is also a maximum
	independent set. Therefore, $\textnormal{corona}(G)\cup A(G)=V(G)$.
	
	By the reasoning of modifying $\textnormal{nucleus}(G)\cup S\cup U$ within $S$,
	we obtain $\textnormal{core}(G)\subseteq\textnormal{nucleus}(G)\cup U$, and
	in the same way $\textnormal{core}(G)\subseteq\textnormal{nucleus}(G)\cup W$,
	hence $\textnormal{core}(G)\subseteq\textnormal{nucleus}(G)$.
	
	Finally,
	\begin{align*}
		\alpha(G) & =\alpha(S)+\left|\textnormal{nucleus}(G)\right|+\left|U\right|\\
		& =\alpha(C(G_{1}))+\dots+\alpha(C(G_{k}))+\left|\textnormal{nucleus}(G)\right|+\frac{\left|C\right|}{2}\\
		& =\frac{\left|C(G_{1})\right|-1}{2}+\dots+\frac{\left|C(G_{k})\right|-1}{2}+\left|\textnormal{nucleus}(G)\right|+\frac{\left|C\right|}{2}
	\end{align*}
	
	\noindent Therefore,
	\begin{align*}
		2\alpha(G) & =\left(\left|C(G_{1})\right|+\dots+\left|C(G_{k})\right|+\left|\textnormal{nucleus}(G)\right|\right)-k+\left|\textnormal{nucleus}(G)\right|+\left|C\right|\\
		2\alpha(G)+k & =\left|D\right|+\left|\textnormal{nucleus}(G)\right|+\left|C\right|.
	\end{align*}
	
	\noindent As we wanted to show. 
\end{proof}

From \cref{11} and \cref{20} we have the following chain
of inclusions in BAB-graphs.
\[
\textnormal{ker}(G)\subseteq\textnormal{core}(G)\subseteq\textnormal{nucleus}(G)\subseteq\textnormal{diadem}(G)\subseteq\textnormal{corona}(G).
\]

\begin{theorem}	[\label{21}\cite{kevin2025RDG}] Let $G$ be an $R$-disjoint
	graph with exactly $k$ disjoint odd cycles. Then
	\[
	\left|\textnormal{corona}(G)\right|+\left|\textnormal{ker}(G)\right|=2\alpha(G)+k.
	\]
\end{theorem}

As shown in \cref{Figura3}, the previous result does not hold
in BAB-graphs. We will show that $2\alpha(G)+k$ is an upper bound for BAB-graphs. This bound is tight,
since a BAB-graph can be an $R$-disjoint graph.

\begin{theorem}
	\label{22} Let $G=(B,G_{1},\dots,G_{k})$ be a BAB-graph,
	then
	\[
	\left|\textnormal{ker}(G)\right|+\left|\textnormal{corona}(G)\right|\le 2\alpha(G)+k=d(G)+\left|G\right|.
	\]
\end{theorem}

\begin{proof}
	By \cref{7}, $N(\textnormal{ker}(G))\subseteq A(G)$, but since
	$\textnormal{ker}(G)\subseteq\textnormal{core}(G)$, we have
	$\textnormal{corona}(G)\cap A(G)\cap N(\textnormal{ker}(G))=\emptyset$.
	Hence, $\textnormal{corona}(G)\cap A(G)\subseteq A(G)-N(\textnormal{ker}(G))$, and then,
	by \cref{18},
	\begin{eqnarray*}
		\left|\textnormal{corona}(G)\cap A(G)\right| & \le & \left|A(G)-N(\textnormal{ker}(G))\right|\\
		& = & \left|N(A(G)-N(\textnormal{ker}(G)))\cap \textnormal{nucleus}(G)\right|\\
		& = & \left|\textnormal{nucleus}(G)-\textnormal{ker}(G)\right|.
	\end{eqnarray*}
	\noindent Finally, by \cref{20},
	\begin{align*}
		2\alpha(G)+k & =\left|D(G)\right|+\left|C(G)\right|+\left|\textnormal{nucleus}(G)\right|\\
		& =\left|\textnormal{corona}(G)\right|-\left|\textnormal{corona}(G)\cap A(G)\right|+\left|\textnormal{nucleus}(G)\right|\\
		& =\left|\textnormal{corona}(G)\right|-\left|\textnormal{corona}(G)\cap A(G)\right|+\left|\textnormal{ker}(G)\right|+\left|\textnormal{nucleus}(G)-\textnormal{ker}(G)\right|\\
		& \ge \left|\textnormal{corona}(G)\right|-\left|\textnormal{corona}(G)\cap A(G)\right|+\left|\textnormal{ker}(G)\right|+\left|\textnormal{corona}(G)\cap A(G)\right|\\
		& =\left|\textnormal{corona}(G)\right|+\left|\textnormal{ker}(G)\right|.
	\end{align*}
	\noindent The second equality follows directly from \cref{7}, since
	\[
	\left|D(G)\right|+\left|C(G)\right|+\left|\textnormal{nucleus}(G)\right|=\left|D(G)\right|+\left|C(G)\right|+\left|A(G)\right|+d(G).
	\]
	\noindent Therefore, $2\alpha(G)+k=d(G)+\left|G\right|.$
\end{proof}

\section{Open Problems}\label{sss5}

\cref{22} motivates the following

\begin{conjecture}
	For every graph $G$, 
	\[
	\left|\textnormal{corona}(G)\right|+\left|\textnormal{ker}(G)\right|\le 2\alpha(G)+k,
	\]
	where $k$ is the number of odd cycles in $G$. Alternatively, $k$ could
	be the number of odd cycles in $D(G)$, or the maximum number of
	disjoint odd cycles.
\end{conjecture}

\begin{problem}
	Characterize graphs with $k$ odd cycles
	such that 
	\[
	\left|\textnormal{corona}(G)\right|+\left|\textnormal{core}(G)\right|>2\alpha(G)+k.
	\] 
\end{problem}

\begin{problem}
	Characterize graphs satisfying 
	\[
	\textnormal{core}(G)\subseteq\textnormal{nucleus}(G)\subseteq\textnormal{diadem}(G).
	\] 
\end{problem}

\begin{problem}
	Find an explicit formula for $\textnormal{corona}(G)$
	in BAB-graphs.
\end{problem}

\begin{problem}
	Find bounds for 
	\[
	\left|\textnormal{corona}(G)\right|+\left|\textnormal{ker}(G)\right|
	\quad \text{and} \quad
	\left|\textnormal{corona}(G)\right|-\left|\textnormal{ker}(G)\right|
	\]
	for arbitrary graphs.
\end{problem}

\section*{Acknowledgments}

This work was partially supported by Universidad Nacional de San Luis (Argentina), PROICO 03-0723, MATH AmSud, grant 22-MATH-02, Agencia I+D+i (Argentina), grants PICT-2020-Serie A-00549 and PICT-2021-CAT-II-00105, CONICET (Argentina) grant PIP 11220220100068CO.

\section*{Declaration of generative AI and AI-assisted technologies in the writing process}
During the preparation of this work the authors used ChatGPT-3.5 in order to improve the grammar of several paragraphs of the text. After using this service, the authors reviewed and edited the content as needed and take full responsibility for the content of the publication.

\section*{Data availability}

Data sharing not applicable to this article as no datasets were generated or analyzed during the current study.

\section*{Declarations}

\noindent\textbf{Conflict of interest} \ The authors declare that they have no conflict of interest.

\bibliographystyle{apalike}

\bibliography{TAGcitasV2025}

\end{document}